\documentclass[12pt,twoside]{amsart}
\usepackage{amsmath}
\usepackage{amssymb}
\usepackage{amscd}
\usepackage{xypic}
\newmuskip\pFqmuskip

\usepackage{amsthm}
\usepackage{graphicx}
\usepackage{enumerate}
\usepackage{graphicx}
\usepackage{tikz-cd}
\usepackage[new]{old-arrows}
\usepackage[cmtip,all]{xy}
\usepackage{amssymb}

\usepackage{graphicx}
\usepackage{tikz-cd}
\usepackage[new]{old-arrows}
\usepackage[cmtip,all]{xy}

\usepackage{amsmath, amsthm, amssymb}

\newtheorem{theorem}{Theorem}[section]
\newtheorem{lemma}[theorem]{Lemma}
\newtheorem{prop}[theorem]{Proposition}
\theoremstyle{definition}
\newtheorem{definition}[theorem]{Definition}
\newtheorem{example}[theorem]{Example}

\newtheorem{corollary}{Corollary}[theorem]

\newtheorem{remark}[theorem]{Remark}

\numberwithin{equation}{section}

\begin{document}

\title{Geometry of weighted Homogeneous spaces}

\email{mrahmati@cimat.mx (M. Reza Rahmati), gflores@cio.mx (G. Flores)}%

\author{Mohammad Reza Rahmati, Gerardo Flores}
\address{}
\curraddr{}
\email{mrahmati@cimat.mx (M. Reza Rahmati), gflores@cio.mx (G. Flores)}
\thanks{}


\subjclass[2010]{ }

\date{}

\dedicatory{}

\begin{abstract}
In this paper, we define the weighted homogeneous space (WHS), denoted by $\frac{G}{P}[\psi_H]$ where $\psi_H$ is weight function defined on the set of simple roots of $G$, by an element $H$ in the highest Weyl chamber. The weight function $\psi_H$ describes the action of the maximal torus $T$ on different Bruhat cells and is well behaved via the change of coordinates defined by the action of the Weyl group $W$. The major effort in this text is to prove basic algebraic and geometric properties of a weighted homogeneous space. The definition can be compared with an existing version given by Reid-Corti \cite{CR}. Additionally, we express $\frac{G}{P}[\psi_H]$ as a whole compact quotient of $G/P$ by a certain action of a finite abelian group. Besides, it is presented a criterion when two WHS with possibly different weight systems are isomorphic. The criteria give a simple method to understand the regular maps between two WHS's, defined by matrices with specific polynomial entries. We also explain invariant K\"ahler differentials on WHS by using certain potential functions on $G$. Our contribution is a generalization of the results presented in \cite{Al, AL, AKQ}. For that, we explain how the weights affect different computations of chern classes of line bundles given in \cite{AL, AKQ}. Finally, we provide a result on the coordinate ring of a WHS by cluster algebras associated to weighted quivers. Specifically, we show that the coordinate ring of a WHS is a weighted cluster algebra of finite type. In this case, the corresponding Dynkin quiver is equipped with a weight function defined on the vertices where the mutations also affect the weights. We present an embedding of a WHS in a product of weighted projective spaces, showing that the coordinate ring is a weighted graded algebra.
\end{abstract}

\maketitle


\section{Introduction} \label{sec:Introduction}
A homogeneous space is the quotient of a reductive Lie group $G$ by a parabolic subgroup $P$, denoted by $X=G/P$. It can also be described as a projective subvariety of the projectivization of some irreducible representation of $G$. A weighted analog of a homogeneous space has been studied by Abe et al. and Reid-Corti \cite{AbM1, AbM2, CR, QAN}. The definition considers an embedding of a weighted homogeneous space (WHS) into a weighted projective space. Their definition adds a high enough number to all the original weights appearing by the torus action to make all the weights positive (because we naturally expect that a WHS should be compact). This paper provides a new definition of weighted homogeneous spaces (WHS) that is different from the previous definition above. Furthermore, we express the weighted homogeneous space as a global compact quotient of ordinary homogeneous space in an alternative result. This is our first contribution after adjusting the definition. In contrast to the definition by Cort-Reid our definition records all the weights as a major invariant of the WHS. It also allows stating the conditions under which two WHS with possibly different weight systems are algebraically isomorphic. 

One may also compare a weighted Homogeneous space to the weighted projective space (WPS). A weighted projective space is an orbifold, i.e., it can be covered by orbifold charts, where each is isomorphic to the quotient of an affine space by a finite group. A weighted projective space is also a toric variety. A comprehensive study of the weighted projective space from a toric point of view can be found in \cite{RT}. A weighted homogeneous space also has a local orbifold structure. A WHS is not, in general, a toric variety for obvious reasons. The maximal dimension of a torus acting on a Homogeneous space (weighted or not) has a strictly smaller dimension than the space. In the paper, we present a fan to explain the weight system in a WHS. Nevertheless, this fan never is a toric fan because of of formal enlargement of its dimension through the definition. It just explains how the weights are set on independent coordinates. 

Besides, we are interested in studying the coordinate ring of a WHS. In a basic proposition we show this ring is a weighted graded homogeneous ring. We present some features of this in Section \ref{sec:WeightedHomogSp}. We also attend a technical approach toward the computation of the coordinate ring of a WHS in Section \ref{sec:CoordRing-WHS} by the theory of cluster algebras associated to quivers. The associated quivers are of finite type and are given by assigning an orientation of a Dynkin diagram. The exchange relations associated to mutation flips can be determined by the Cartan matrix of the Lie group $G$.
\subsection{Motivation}
The original idea behind this work is some analogy between the projective spaces and homogeneous spaces. The same way we consider a weighted projective space as a kind of generalization to the projective space, we may also wish to look for a way to define the weighted homogeneous variety. In some way, we have tried to build up the WHS from the model of WPS. A version of the definition for the WHS exists in the literature initially due to Abe et al. \cite{AbM1, AbM2}, and also by Corti-Reid \cite{CR}. They first consider a $\mathbb{C}^*$-fibration, which lifts a homogeneous space $G/P$, to the $Cone (G/P) \smallsetminus 0$ and then try to assign some positive weights to the coordinates on the cone. In their analysis, the assignment of the weights on different charts can produce negative weights on the coordinates.
For this reason, they oblige to add a sufficiently big integer $u \gg 0$ to all the weights on all the charts. In this way, the data of the weights and its relation to the combinatorial Lie theory of the Lie group $G$ is lost in their method. One of our significant efforts in this text was to fix this problem. In fact, we make a different definition for the weight system and will fix the problem above through some explanation of the Weyl group action on the Bruhat cells. Moreover, we will explain that the problem of arising negative weights is be solved in our definition. Our first result expresses the WHS $\frac{G}{P}[\psi_H]$ as the whole compact quotient of the homogeneous space $G/P$. This automatically resolves the question of well definedness of the weights and its direct affect on the compactness of a WHS. 

The data of the weights on different coordinate access and their connection to the root system of the Lie group $G$ is crucial in our approach. We want to find a classification theorem to determine when two weight systems on WHS define the same geometric object. Specifically we present a basic combinatorial criterion to characterize when two WHS $\frac{G}{P}[\psi_{H_1}]$ and $\frac{G}{P}[\psi_{H_2}]$. One may also study further that how usual geometric tools such as differential forms and regular maps are  defined on WHS. Our construction naturally allows to determine when there can exist well defined algebraic maps $\frac{G}{P}[\psi_{H_1}] \to \frac{G}{P}[\psi_{H_2}]$, in terms of the weights. Although this can also be done in the definition by Corti-Reid in \cite{CR}, but the data of the original weights is not accessible in their definition, nor even the integer parameter $u$ is not definitely defined.

The WHS $\frac{G}{P}[\psi_{H_1}]$ has a natural embedding into a weighted projective space defined as the canonical embedding of an ample divisor.  
In the way we define a WHS, we can also embed it in a product of weighted projective spaces obtained from highest weight representations of the Lie group $G$ with highest weight to be its fundamental weights. Then we can embed it in a single projective space via the Segre embedding. However, for our purpose, we do not insist on embedding it in a single projective space. In fact, in this way, we can get a simple understanding of the coordinate ring of a WHS as a weighted graded algebra. By \cite{GLS, Sc} the coordinate ring of any partial flag variety is a cluster algebra, one may expect to prove that the coordinate ring of a WHS is a weighted cluster algebra. The complexity of the definition of a cluster algebra forces a more detailed explanation of this concept. One approach to cluster algebras is by quivers. In this case, the mutations of the cluster variables are defined through the mutations of the quiver. Cluster algebras classify as finite or infinite types. In the case of a homogeneous space, we are involved with quivers that come from the Dynkin diagram of a Lie group $G$. So they have finite type. A natural expectation is to describe the coordinate ring of the weighted homogeneous space $\frac{G}{P} \left [\psi_H \right ]$ by doing mutations on a weighted quiver. There are several definitions of weighted quivers in the literature. We employ a definition given in \cite{OT} for weighted quivers which are equipped with a weight function $w:Q \to \mathbb{Z}$ on its vertices. The difficulty here is to understand the effect of mutations on the weights. 

\subsection{Related Works}
A definition for the weighted Grassmannian was first introduced and studied by Corti-Reid \cite{CR}, following the work of Grojnowski. As a projective variety, it has at worst orbifold singularity and a torus action. In their approach, WHS is defined in a weighted projective space by the wellknown Pl\"ucker relations. A systematic study of the cohomology of the weighted homogeneous space via the Schubert calculus is given in \cite{AbM1, AbM2}, and \cite{QAN}. Some of the techniques of these works will also apply to our setup without changes. However, as was mentioned in the introduction section, we will not follow their definition in this paper. The major effort in this text is to provide a new definition for the WHS, such that the weights can be connected to the combinatorial structure of the homogeneous space. In fact, we need the information on the weights as part of the geometry of a WHS. In studying the coordinate ring of a WHS, via the weighted quivers the mutations will exchange the weights in the same manner as that the action of the Weyl group on Bruhat cells will make a change of coordinates on the homogeneous space. 
\subsection{Contribution}\label{subsec:contribution}
The main contribution of this paper is a new definition of weighted homogeneous space (WHS). With that aim, we provide two major new definitions for a WHS. The first definition is by a non-compact quotient through the action of a 1-parameter subgroup on the affine cone over a homogeneous space $G/P$. In contrast, the second definition is by a compact quotient of the projective variety $G/P$. The two definitions are compatible, and we prove this in a commutative triangle, [see \eqref{eq:diagram_first} below]. The explanations are as follows. In Section \ref{sec:WeightedHomogSp} we present a notion of weighted homogeneous spaces $\frac{G}{P} \left [\psi_H \right ]$, where $G$ is a complex semisimple Lie group, $P$ is a parabolic subgroup of $G$, $H$ is an element of highest Weyl chamber, and $\psi_H$ is the weight system defined by $H$, [see also Section \ref{sec:WeightedHomogSp} below]. Notice that a weight system gives rise to an action of the maximal torus on the Bruhat cells. Let $T$ be a maximal torus of $G$, $\mathfrak{h}$ be the corresponding Cartan subalgebra of $\mathfrak{g}=Lie(G)$, and $R=R_- \cup R_+$ is the associated root system, while $\Pi$ are the simple roots for $\mathfrak{g}$. Suppose 
\begin{equation}
C_+=\{ H \in \mathfrak{h}_{\mathbb{R}} | \langle \alpha, H \rangle  \geq 0, \forall \alpha \in \Pi\}    
\end{equation}
is the highest Weyl chamber. Assume $H \in C_+$ is so that $\langle -,H \rangle $ is a dominant integral weight. The weight system is given by an integral dominant weight function $\psi_H:\Pi \to \mathbb{Z}_{\geq 0},\ H \in C_+$. The map $\psi_H$ defines a 1-parameter subgroup of $G$ and determines an action of the maximal torus $T$ of $G$ on the Bruhat cells of $G/P$, [see section \ref{sec:WeightedHomogSp}]. Specifically, we define the weighted homogeneous space
as
\begin{equation} \label{eq:WHS}
\frac{G}{P} \big [\psi_H \big ]=\left [Cone \left (\dfrac{ G}{P} \right )\diagdown 0 \right ] \big /_HT .
\end{equation} 
The action of the maximal torus $T$ is defined on each open Bruhat cell separately. The precise meaning of the quotient in \eqref{eq:WHS} is through the action of $T$ on the Bruhat cells [see section \ref{sec:WeightedHomogSp}]. 

Weighted homogeneous spaces (WHS) are in many senses related to ordinary homogeneous spaces $G/P$. IN this sense, we have a locally finite quotient map $\pi_H: \frac{G}{P} \longrightarrow \frac{G}{P} \left [\psi_H \right ]$, where the map $\pi_H$ is continuous in quotient topology. However, it has critical values at the singularities of $\frac{G}{P} \left [\psi_H \right ]$. The singularities are finite quotient singularities and provide the structure of an orbifold on WHS. The quotient map $\pi_H$ can be explained on the Bruhat cells, as affine charts of $G/P$. The Bruhat decomposition of $G/P$ descends to the one for $\frac{G}{P} \left [\psi_H \right ]$ by the map $\pi_H$. In other words, we briefly have
\begin{equation}
\frac{G}{P} \big [\psi_H \big ]=\coprod_{w \in W/W_L} X_w[\psi_H], \qquad \pi_H:X_w \to X_w[\psi_H]
\end{equation}
where $X_w$'s are the Bruhat cells [see Sections \ref{sec:HomogSp} and \ref{sec:WeightedHomogSp} for notations]. Compared to standard homogeneous manifolds, a WHS behaves very much like the weighted projective spaces (WPS) compared to projective spaces. 

We provide a second definition for the WHS $\frac{G}{P} [\psi_H ]$ by expressing it as a quotient of the projective variety $G/P$ by a finite Abelian group. In this case, the underlying topological space of $\frac{G}{P} [\psi_H ]$ is automatically compact. We also show that the two definitions are compatible by proving the commutativity of the diagram,
\begin{equation}\label{eq:diagram_first}
\begin{tikzcd}
 & G/P \arrow{dr}{/ \times_{\alpha}\mathbb{Z}_{\langle \alpha, H \rangle} } \\
\left [\mathcal{C}one \left (\dfrac{G}{P} \right )\smallsetminus 0 \right ]\arrow{ur}{/\mathbb{C}^*} \arrow{rr}{/[\psi_H]} && \dfrac{G}{P} \left [\psi_H \right ],
\end{tikzcd}
\end{equation}
which can be compared with its analog for weighted projective spaces (WPS) [see remark \ref{WPS-Triangle}]. The horizontal map above is the projection that we obtain through the definition of the weighted homogeneous space $\frac{G}{P} \left [\psi_H \right ]$, i.e., by \eqref{eq:WHS}. The map on the left edge of the triangle is the cone construction of the classical algebraic geometry, and it exists because the homogeneous space $G/P$ is by definition a projective variety. Finally, the edge map at the right is the map that its existence is proved in the theorem \ref{th:theorem32}.

Certain ingredients imply two WPS of the same type but with probably different weight systems are isomorphic, [\cite{RT}, see Theorem \ref{th:theoremWPS} below]. For instance, we may permute the weights on the torus coordinates. Another ingredient is when we multiply all the weights by an arbitrary positive integer. In this case, the resulting weight system does not produce a new variety. In other words, the resulting variety is algebraically isomorphic to the original WPS. We can always reduce the weights in this manner to obtain reduced weight systems, [see Section \ref{sec:WeightedHomogSp} for exact definition]. Then, we can ask a refined question when two reduced weight systems produce two isomorphic WPS's. One way to answer this question is by using the tools in toric geometry since a WPS is a toric variety. One way to compare two WPS is to compare their toric fan. It has been proved in \cite{RT} that two reduced WPS with two weight systems are isomorphic if and only if the two fans can be transformed to one another by a matrix in $GL_n\mathbb{Z}$. We generalize this argument to weighted homogeneous manifolds. To a WHS $\frac{G}{P} \left [\psi_H \right ]$ we associate a combinatorial data $\Sigma_H$. The association of a $\Sigma_H$ to a WHS is a formal way to record the weights in a convex polytope or coordinate axes of a lattice. This lattice is a formal enlargement of the toric fan of the maximal torus of $G$. For this it is no longer a toric fan but it reserves the data of the torus action on all coordinates. We prove an analogous result that two WHS with reduced weight systems are isomorphic if and only if there exists a matrix in $GL_N(\mathbb{Z})$ that transforms the formal fan of one into the formal fan of the other. The only difference with WPS is that we can not deduce this from toric geometry. In order to study the isomorphism question for WHS's we consider relative maps $f:\frac{G}{P}[\psi_{H_1}] \to \frac{G}{P}[\psi_{H_2}]$. The map exists only if there is a homomorphism $\beta_f:\Sigma_{H_1} \longmapsto \Sigma_{H_2}$. One of the main results of this paper is a combinatorial criteria that two weighted homogeneous spaces possibly with different weight systems are isomorphic. 

In Section \ref{sec:InvKahler} we study invariant k\"ahler forms on a WHS. The concept is central in the theory of homogeneous spaces. The K\"ahler forms appear as certain Chern classes of line bundles on $G/P$. when $K$ is a maximal compact subgroup of $G$, the $K$-invariant K\"ahler forms represent the curvature of $K$-invariant line bundles. The concept is introduced in \cite{AKQ}. We want to introduce this notion on a weighted homogeneous space; for that, we first classify the $K$-invariant K\"ahler forms on $\frac{G}{P} \left [\psi_H \right ]$. Our result is a generalization of that, which is already presented in \cite{AKQ}.

Notice that the K\"ahler differentials on WHS $\frac{G}{P} \left [\psi_H \right ]$ can be defined by pluri-subharmonic potentials 
\begin{equation} \label{eq:Potential}
\begin{aligned}
\phi: \ &G \to \mathbb{R} \cup - \infty\\
\phi(g)&=\sum_{J(S)} c_{\alpha}\log \| \rho_{\alpha}(g).v_{\alpha} \|. 
\end{aligned}  
\end{equation}
The representation $\rho_{\alpha}$ is the irreducible representation of $G$ of highest weight $w_{\alpha}$,  the fundamental weights of $G$, and $v_{\alpha}$ is the highest weight vector. The K\"ahler forms on $\frac{G}{P} \left [\psi_H \right ]$ are pulled back to the K\"ahler forms on $G/P$, via the locally quotient map $\pi_H$. We shall explain in Section \ref{sec:InvKahler} that how the equation \eqref{eq:Potential} may be compared to the non-weighted case [see Theorem \ref{Thm:InvariantForm}]. Thus, as partial contributions, we generalize some results in \cite{AKQ}.

Finally, we discuss the coordinate ring of homogeneous varieties as a cluster algebra and its mutations. There exists embedding of a WHS into a product of weighted projective spaces on highest weight representations, where we obtain a multi-graded algebra structure on the coordinate ring of $\mathbb{C}\left [\frac{G}{P} \left [\psi_H \right ]\right ]$. The coordinate ring of a Homogeneous space is a finite extension of the weighted one. The degree of the extension is equal to the product of the weights at all the roots involved in the quotient $G/P$. It is known that the coordinate ring of $G/P$ is a cluster algebra, \cite{GLS}, \cite{FZ1}, \cite{FZ5}. These cluster algebras are of finite type, i.e., the adjacency matrix is defined via a Cartan matrix. An example of it is the Grassmann cluster algebras, i.e., the coordinate ring of the Grassmannian $Gr(k,n)$, where cluster complex can be obtained as a dual to a $(k,n)$-diagram on the disc. One approach to cluster algebras is through quivers and their mutations. 

Following a definition of weighted quivers in \cite{OT} we explain that the coordinate ring of a WHS is a weighted cluster algebra defined by a weighted quiver of finite type. This appears as the last contribution of the text. 
\subsection{Organization of the text}
The remainder of this paper is organized as follows. Section \ref{sec:Introduction} presents the introduction and explanations of an overall view of the paper. We also introduce some of the notations in the text. Section \ref{sec:HomogSp} is classical and reviews the well-known facts on homogeneous spaces. The references are \cite{M, Sp, Stei, Ste2, Ful1, Bri, Bor, Dy, K, Kn, B, Ho, Sa} as well as various other classical texts in Lie algebras and their representations. In Section \ref{sec:WeightedHomogSp} we introduce the definition of the weighted homogeneous space, and also it contains the main contributions of the text. In section \ref{sec:InvKahler} we study the invariant K\"ahler forms on a WHS. The results are generalizations of analogous theorems in \cite{AKQ}. For the basics, the reader can consult with \cite{Br, CM, Ok, Al, C, AKQ, KM, W}. Section \ref{sec:Prel-ClusterAlg} explains the basics of cluster algebras. The materials of this section are well known. The standard references are \cite{BGH, CFZ, FZ1, FZ2, FZ3, FZ4, Le, Lu1, Lu2, MRZ, MS, Ze, BFZ, FR, FZ5, FZ6}. In section \ref{sec:CoordRing-WHS}, we discuss weighted quivers and the coordinate ring of weighted homogeneous spaces. Finally, in section \ref{sec:conc} some conclusions and final remarks are given. We highlight that the main contributions of the text are given in Sections  \ref{sec:WeightedHomogSp}, \ref{sec:InvKahler} and \ref{sec:CoordRing-WHS}.
\section{Homogeneous Spaces}\label{sec:HomogSp}
The material of this section is well known, the references are \cite{Ful1, Bor, Stei, Sp, B}. Assume $\mathfrak{g}= Lie(G)$ is a semisimple Lie algebra, $\mathfrak{h}=\langle H_{\alpha_1},...,H_{\alpha_n} \rangle$ be a Cartan and $R=R_+ \bigcup R_-$ the associated root system with $\Pi=\{ \alpha_1,...,\alpha_n \}$ the simple roots. Let
\begin{equation} \label{eq:root-decom}
\mathfrak{g}=\mathfrak{h} \oplus \ \bigoplus_{\alpha \in R}\mathfrak{g}_{\alpha}
\end{equation}
be the root decomposition, and let  
\begin{equation}
\mathfrak{b}=\mathfrak{h} \oplus \bigoplus_{\alpha \in R_+}\mathfrak{g}_{\alpha}
\end{equation}
the maximal solvable (Borel) subalgebra with corresponding Lie group $B$. Then $B$ is a closed subgroup, and $G/B$ is a compact topological space. In fact, $G/B$ can be embedded in a Grassmannian, where $B$ is the stabilizer of $\wedge_{\alpha \in R_+} \alpha$ for the action 
\begin{equation}
G \times \mathbb{P}(\Lambda^{|R_+|}\mathfrak{g}) \to \mathbb{P}(\Lambda^{|R_+|}\mathfrak{g}) .
\end{equation}

A closed connected subgroup of $G$ is parabolic if $G/P$ can be realized as the orbit of the action of $G$ on $\mathbb{P}(V)$ for some representation $V$ of $G$. Thus $G/P$ is a projective variety. An example of a Parabolic subgroup is a Borel subgroup $B$, and a parabolic subgroup contains some Borel. In other words, Borels are minimal parabolic subgroups. For the corresponding Lie algebras we have $\mathfrak{b} \subset \mathfrak{p} \subset \mathfrak{g}$. Therefore 
\begin{equation}
\mathfrak{p}=\mathfrak{h} \oplus \bigoplus_{\alpha \in \Delta}\mathfrak{g}_{\alpha}, \qquad \Delta \supset R_+
\end{equation}
and for $\mathfrak{p}$ to be a subalgebra, $\Delta$ must be closed under addition. It follows that the parabolic subgroups of the simple group $G$ are in 1-1 correspondence with subsets of nodes of the Dynkin diagram, i.e., simple roots. 

Assume $V_{\lambda}$ is the irreducible representation of $G$ with the highest weight $\lambda$. Consider the action of $G$ on $\mathbb{P}(V_{\lambda})$, and let $p \in \mathbb{P}(V_{\lambda})$ be the point corresponded as the subspace of weight $\lambda$. One can show that $G.p$ is the unique closed orbit of the action of $G$ and $P_{\lambda}$ is the parabolic corresponded to the subset of simple roots perpendicular to $\lambda$, i.e., the faces of Weyl chambers. Therefore there is a 1-1 correspondence $\{ \text{parabolic subgroups of}\ G\} \ \stackrel{1-1}{\longleftrightarrow} \  \{\text{faces of the Weyl chambers}\}$ where the face corresponding to $P_{\lambda}$ is that containing $\lambda$. The weights in interior of Weyl chamber correspond to $B$. The Weyl group can be written as $W=N_G(T)/T$ where $N(T)$ is the normalizer of $T$ in $G$. For $w \in W$ one can choose a representative $\dot{w}$. Besides, one can show that the double coset $B\dot{w}B$ is independent of the lift. The Bruhat decomposition of $G$ is the double coset decomposition corresponding to $B\setminus G/B$. The latter introduces the following important theorem.
\begin{theorem} (Bruhat decomposition) \cite{Ful1}  \label{th:BruhatdecompG}
We have a disjoint union that follows:
\begin{equation}
G=\coprod_{w \in W} B\dot{w}B   .
\end{equation}
\end{theorem}
The decomposition in Theorem \ref{th:BruhatdecompG} above is the double coset decomposition of the pair $(B,B)$ in $G$. The components $B\dot{w}B$ are called Bruhat cells. The Weyl group $W$ acts on the Bruhat cells. We give a brief analysis of this action in the following. Let $U$ (resp. $U^-$) be the unipotent radical of $B$ (resp. opposite Borel $B^-$). Define 
\begin{equation} 
U_w=U \cap \dot{w}U^- \dot{w}^{-1}, \qquad U_w'=U \cap \dot{w}U\dot{w}^{-1}
\end{equation} 
Again it is independent of the lift. The Lie algebra of $U_w$ is the sum of all root spaces $\mathfrak{g}_{\alpha}$ such that $\alpha$ is positive but $w^{-1}\alpha$ is negative. For $U_w'$ it is the sum of those where $\alpha$ and $w^{-1}\alpha$ are both positive. Thus $U_wU_w'T=B$. Replacing this identity, one gets 
\begin{equation} 
B\dot{w}B=U\dot{w}B.
\end{equation}
It is not hard to check that the presentation of an element in $B\dot{w}B$ in the form of the above product $u .\dot{w}. b$ is unique. It is also well known that the Bruhat cells have an affine structure. The following proposition provides its dimension.
\begin{prop} \label{th:Bruhatcell} \cite{Ful1}
The dimension of $U_w$ is the cardinality $R_+ \cap w(R_-)$. In fact, $U_w =\mathbb{C}^{l(w)}$ is affine, where $l(w)$ is the length of $w \in W$.
\end{prop}
A consequence of Theorem \ref{th:BruhatdecompG} and Proposition \ref{th:Bruhatcell} is the following.
\begin{theorem} (Bruhat decomposition of $G/B$) \cite{Ful1} \label{th:BruhatdecompB} 
We have a disjoint union of cosets 
\begin{equation} 
G/B=\coprod_w B\dot{w}B/B, \qquad X_w= B\dot{w}B/B
\end{equation}
where $X_w \cong \mathbb{C}^{l(w)}$ are Bruhat cells, and $l(w)=\sharp \{ \ \alpha \in R_+ \ |\ w\alpha <0\ \}$. 
\end{theorem} 
Bruhat decomposition provides a 1-1 correspondence between 
\begin{equation}
B\setminus G/B \ \ \longleftrightarrow \ \ W=N_G(H)/H
\end{equation}
which is the origin to define different kinds of orders on the Bruhat cells. One can also express the Bruhat decomposition for $P\setminus G/B$. One has 
\begin{equation}
P=\coprod_{w \in W_L}B\dot{w}B    
\end{equation}
where $L$ is the Levi component in Levi decomposition $P=LN$ ($N$ is unipotent), and $W_L=N_L(T)/T$. Now we can state the Bruhat decomposition more generally for $G/P$ where $P$ is a parabolic subgroup.
\begin{theorem} (Bruhat decomposition of $G/P$) \cite{Ful1} \label{th:Bruhat-decom}
We have the decomposition 
\begin{equation}
\ P\dot{w}B=\coprod_{x \in W_L}BxwB.
\end{equation} 
It follows that; 
\begin{equation}\label{eq:gpok}
G/P=\coprod_{w \in W/W_L} P \dot{w} B.
\end{equation}
\end{theorem}
Notice that the Bruhat decomposition in Theorem \ref{th:Bruhat-decom} above provides a 1-1 correspondence as
\begin{equation}
P\setminus G/B  \ \longleftrightarrow \ W(P)=W/W_L=N_L(T)/T.
\end{equation}
There is a cross-action of the Weyl group on $P\setminus G/B$ that is generated by the following action of simple reflections on the Bruhat cells
\begin{equation}
s_{\alpha} \times PwB \longmapsto Pws_{\alpha}^{-1}B , \qquad \alpha \in \Pi   .
\end{equation}
We give a basic description of Bruhat decomposition for $SL_n \mathbb{C}$ and $Sp_{2n} \mathbb{C}$ in the following example.
\begin{example} \label{ex:Bruhat-decom}
Consider $G=SL_m\mathbb{C}$, with standard Borel $B$ of upper triangular matrices. We have $W=S_m$ the symmetric group on $m$ elements, $U_w$ are matrices with $1$ on the diagonal and $0$ in the $(i,j)$ coordinates where $i >j,\ w^{-1}(i) <w^{-1}(j)$ and its dimension equals 
\begin{equation} 
l(w)=\sharp\{\ (i,j)\ |\ i >j,\ w(i) <w(j)\ \}.
\end{equation} 
The Bruhat cell $X_w$ of $SL_m\mathbb{C}/B$ consists of the flags such that the dimension of the intersections with standard flag are governed by $w \in S_m$. 

The Bruhat decomposition of $Sp_{2n}\mathbb{C}$ is obtained from that of $SL_{2n}\mathbb{C}$ provided a change of the order on basis $e_1,...,e_n,e_{2n},...,e_{n+1}$. Then $B$ consists of matrices with upper left block upper triangular and lower right block is lower triangular, with $0$ other places. Its Weyl group is 
\begin{equation} 
W=\{ \ w \in S_{2n} \ |\ w(n+i)=w(i)+n\ \}.
\end{equation} 
The Bruhat decomposition of $SO_n \mathbb{C} $ has a similar explanation. 
\end{example} 
Finally, we explain the Bruhat decomposition of the double coset $K \setminus G/B$ where $K$ is a maximal compact subgroup, and $B$ is a Borel. In this case, we need to consider a real form of the Lie group $G$. To be more precise, let $\theta:G \to G$ be a Cartan involution of $G$, such that $K=G^{\theta}$ and $B=TU$ a $\theta$-stable Borel of $G$. Suppose 
\begin{equation}
\tilde{V}= \left \{ x \in G\ |\ \frac{\theta(x)}{x} \in N(T)\ \right \}.
\end{equation} 
There is an action of $K \times T$ on $\tilde{V}$ by,
\begin{equation} 
(K \times T) \times \tilde{V} \to \tilde{V}, \qquad (k,t)\times x \mapsto kxt^{-1}.
\end{equation}
Let $V$ be the orbit space of the action, and denote a representative of $v$ as $\dot{v}$, \cite{PSY}; thus, the following theorem presents the coset decomposition for $K \setminus G/B$. 
\begin{theorem} (T. A. Springer \cite{Sp}) \label{th:Springer}
There is a bijection 
\begin{equation}
V \stackrel{1-1}{\longleftrightarrow} K \setminus G/B    .
\end{equation}
\end{theorem}
The cross action of $W$ on $K\setminus G/B$ is given by
\begin{equation}
s_{\alpha} \times K\dot{v}B \mapsto K\dot{v}s_{\alpha}^{-1}B
\end{equation}

A similar formula can be stated for the Parabolic $P$, and thus, we have the following more general coset presentation for $K \setminus G/P$ given in the following theorem, [see \cite{Ye} for details].
\begin{theorem} \cite{PSY} \label{thm:K-G-P}
There is a bijection 
\begin{equation}
K \setminus G/P \stackrel{\cong}{\longrightarrow} \left ( L^{\theta}\setminus L/P \cap L \right ) \times_{W_L} W    
\end{equation}
where the fiber product is with respect to the cross action of $W_L$ on $L^{\theta}\setminus L/P \cap L$. 
\end{theorem}
We shall consider the Bruhat cells as the manifold charts for $G/P$ and also for $\frac{G}{P} \left [\psi_H \right ]$. First, it would be essential to understand how the coordinate changes are formulated by the action of the Weyl group on the Bruhat cells. The cross action on the Bruhat cells is the basic building block of these coordinate changes.
%
%
\section{Weighted Homogeneous Spaces}\label{sec:WeightedHomogSp}
This section defines a weighted homogeneous space (WHS) and then proves some of its fundamental geometric properties. We give two equivalent definitions that one appears in the Theorem \ref{th:theorem32}. The first definition is by the $\mathbb{C}^*$-action on the affine cone over $G/P$. The central point in this definition is that the weights on the coordinates of a specific chart are positive, so the resulting quotient space is compact. The second definition is to express a WHS as a whole compact quotient of $G/P$ by a finite Abelian group. Because our definition is new, all the material in this section has new contributions, except the Theorem \ref{th:theoremWPS}.

A weighted homogeneous space (WHS) should be understood as a homogeneous space $G/P$ where the coordinate variables have positive integer weights. If we consider a projective space instead, then a weighted projective space is a projective space where its coordinates have positive integer weights. However, the coordinate system of a homogeneous variety is more complicated than a projective space. To explain the coordinate system on homogeneous spaces, we first consider the case $G/B$ where $B$ is a Borel subgroup of $G$. For the sake of clearance, we consider the Bruhat cells as coordinate charts of $G/B$. The action of the maximal torus $T$ of $G$ on the local coordinates of the Bruhat cells can be easily explained through 1-parameter subgroups of $G$. The case of $G/P$ for arbitrary parabolic subgroup $P$ can be described similarly. 
\subsection{The definition of a WHS} 
Let $G$ be a complex semisimple reductive Lie group and $P$ a parabolic subgroup and $B$ a Borel subgroup of $G$. Set $\mathfrak{g}=Lie(G)$, a complex semisimple Lie algebra, $\mathfrak{h}$ a Cartan subalgebra, and $R=R_+ \cup R_-$ the associated root system, and $\Pi$ the simple roots. Set 
\begin{equation} 
\mathfrak{h}_{\mathbb{R}}=\{\ H \in \mathfrak{h}\ |\ \alpha(H) \in \mathbb{R}, \ \forall \alpha \in R\ \}
\end{equation}
and $C_+$ the highest Weyl chamber, i.e., 
\begin{equation} 
C_+=\{H \in \mathfrak{h}|\alpha(H) \geq 0, \forall \alpha \in R_+\}.
\end{equation} 
For any $H\in \mathfrak{h}$ define $\psi_H:\Pi \to \mathbb{R}, \ \alpha \mapsto \alpha(H)$. The weight system is given by the integral dominant weight function
\begin{equation} 
\psi_H:\Pi \to \mathbb{Z}_{\geq 0},\qquad H \in C_+ .
\end{equation} 
Then $\psi_H$ assigns a positive weight to each simple root i.e., the nodes of the Dynkin diagram. 

Let $ \ \{ X_{\alpha}, H_{\alpha}, X_{-\alpha} \}$ be an $\mathfrak{sl}_2$-triple in $\mathfrak{g}$. The root groups are defined via the maps 
\begin{equation} \label{eq:root-group}
\phi_{\alpha}:SL_2\mathbb{C} \to L_{\alpha}, \qquad Lie(L_{\alpha})=\langle X_{\alpha}, H_{\alpha}, X_{-\alpha} \rangle .
\end{equation} 
On $G$ we shall consider the following:
\begin{eqnarray}
z_{\alpha}(z)=\phi_{\alpha}\begin{pmatrix}
1 & z\\
0 & 1
\end{pmatrix}, \ z_{-\alpha}(z)=\phi_{\alpha}\begin{pmatrix}
1 & 0\\ \nonumber
z & 1
\end{pmatrix}, \\
 \check{\alpha}(z)=\phi_{\alpha}\begin{pmatrix}
z & 0\\
0 & z^{-1}
\end{pmatrix} .
\end{eqnarray}

Set $U$ to be the unipotent radical of the Borel $B$. Let $T$ be the maximal torus of $G$. There is a $T$-equivariant isomorphism $U \cong \mathbb{C}^{|R_+|}$, where $T$ acts on $\mathbb{C}^{|R_+|}$ by 
\begin{equation} 
t.\ x_{\alpha}=\alpha(t)\ x_{\alpha}, \qquad \alpha \in R_+ .
\end{equation} 
Thus a $1$-parameter subgroups gives rise to an action on $U$ defined by,
\begin{equation}\label{eq:torus-action}
t._H\ x_{\alpha}=t^{\ \langle \alpha,H\rangle}\ x_{\alpha}, \qquad \alpha \in R  . 
\end{equation}
The Weyl group $W=N_G(T)/T$ acts on $T$ by conjugation and the action of $W$ on the weight lattice $\mathcal{P}$ is given by 
\begin{equation} \label{eq:weylaction}
t^{w(\lambda)} =(w^{-1}tw)^{\lambda}  , \qquad w \in W, \ \lambda \in \mathcal{P}, \ t \in T  
\end{equation}
where $\lambda$ is regarded as a character of the torus $T$ by
\begin{equation} 
\lambda(\exp(t))=\exp(t)^{\lambda}=e^{\lambda(t)}. 
\end{equation}
The Weyl group also acts on the cells $U$, so that it permutes the Bruhat cells (cross action). Under this change of coordinates, the $T$-action gets twisted as 
\begin{equation} \label{eq:Weyl-action}
t._H\ x_{\alpha}^w=\ t^{w(\alpha)} . \ x_{\alpha}^w, \qquad \alpha \in R_+, \ x_{\alpha}^w \in w.U
\end{equation}
where $x_{\alpha}^w$ states the coordinate $x_{\alpha}$ in $w U$. Remember that $H$ was chosen so that $\langle -,H \rangle$ acting as an integral dominant weight, thus we get all the above weights positive. The identity \eqref{eq:Weyl-action} is crucial in our approach to define the weighted homogeneous space. In fact, this relation guarantees that the weights of the action on different Bruhat cells are positive, so we get a compact space at the end. Without attending to this identity, one may lead to the former definition given by Abe et al. and Corti-Reid, \cite{AbM1, AbM2, CR, QAN} to add a big positive integer $u \gg 0$ to the weights to make them all positive. This process will disturb the data recorded in the weight system. However, one can obtain the embedding into the weighted projective space at the end. The equation \eqref{eq:Weyl-action} tells the way the weights are exchanged through the Weyl group action on the Bruhat cells. Because this action is transitive, we get a complete description of the weights on different manifold charts. More generally, the above analogy recommends the following. We define the weighted homogeneous space by
\begin{equation} \label{eq:weightedG/B}
\frac{G}{B} \big [\psi_H \big ]=\left [Cone \left (\dfrac{ G}{B} \right )\diagdown 0 \right ] \big /_HT.
\end{equation}
The meaning of the quotient in \eqref{eq:weightedG/B} has to be understood by \eqref{eq:torus-action} and this action is well-defined because of \eqref{eq:weylaction}. One can generalize this definition for a general parabolic, $P$. That introduces our first main definition.
\begin{definition} [Main Result] \label{def: WHS}
Assume $\psi_H \in Hom(\Pi,\mathbb{Z}_+)$ be a positive dominant integral weight corresponding to $H \in C_+$. Define
\begin{equation} 
\frac{G}{P} \big [\psi_H \big ]=\left [Cone \left (\dfrac{G}{P} \right )\diagdown 0 \right ] \big /_HT
\end{equation}
where the action of $T$ is given by \eqref{eq:torus-action} and \eqref{eq:Weyl-action}.
\end{definition}
Assume the parabolic subgroup $P$ corresponds to the set of simple roots $\Pi \setminus I=J$. Let $R(P)$ (resp. $R(P)_+$) be the roots (resp. the positive roots) with support in $J$ and let $U(P)$ be the Bruhat cell corresponding to the identity. Then, $U(P) \cong \mathbb{C}^{|R(P)_+|}$. The $T$-action on the coordinates $x_{\alpha}$ are given by the same identity as in \eqref{eq:torus-action}. The actions of the Weyl group behave as a change of coordinates. The first question that we answer is the possibility of expressing the weighted homogeneous space as a compact quotient. It appears as our first result next.
\begin{theorem} [Main Result]\label{th:theorem32}
We have 
\begin{equation}\label{eq:gPpsih:th}
\frac{G}{P} \left [\psi_H \right ]=\left [ \times_{\alpha \in R(P)_+} \mathbb{Z}_{\langle \alpha,H \rangle} \right ] \setminus G/P  .  
\end{equation}
\end{theorem}
\begin{proof} We highlight that the claim in \eqref{eq:gPpsih:th} is to express the WHS as the global quotient of the projective variety $G/P$ by the finite group $\left [ \times_{\alpha \in R(P)_+} \mathbb{Z}_{\langle \alpha,H \rangle} \right ]$. 
We first explain how the following diagram is well-defined and commutative:
\begin{equation}\label{eq:diegramcone}
\begin{tikzcd}
 & G/P \arrow{dr}{/  \times_{\alpha \in R(P)_+} \mathbb{Z}_{\langle \alpha,H \rangle }} \\
\left [Cone \left (\dfrac{ G}{P} \right )\diagdown 0 \right ] \arrow{ur}{/\mathbb{C}^*} \arrow{rr}{/[\psi_H]} && \dfrac{G}{P} \big [\psi_H \big ],
\end{tikzcd}
\end{equation}
where the left-hand edge map of the triangle is the ordinary cone over a projective variety, i.e., a $\mathbb{C}^*$-fibration; the horizontal map is the definition of the WHS, i.e., \ref{def: WHS}. The right-hand edge map is the map we wish to prove its existence by this theorem and is the quotient map by a finite Abelian group. Because the charts are compatible with the left edge cone construction, the commutativity means that the horizontal quotient map induces the right-hand edge map. We show the existence of the right edge quotient map, which will prove the theorem.

In order to show the isomorphism of the two projective varieties on both sides of \eqref{eq:gPpsih:th}, one can prove that their corresponding homogeneous coordinate rings are isomorphic, [the same method has been used in \cite{Do} for weighted projective spaces]. Now, we note that the argument used in \cite{Do} (please refer to Sec. 1.2.2 of that reference) extends over here. The action of the group scheme $\mu= [\times_{\alpha \in R(P)_+} \mu_{\langle \alpha,H \rangle}]$ on $G/P$ is induced by the action of $\mu$ on $\mathbb{C}\left [\frac{G}{P} \right ]$ and is defined by the $x_{\alpha} \mapsto x_{\alpha} \times \overline{Y_{\alpha}}$ where $\overline{Y_{\alpha}}=Y_{\alpha} \ \mod (Y^{\langle \alpha, H \rangle}-1)$. This demonstrates that the induced map defined by
\begin{equation}
\phi: \mathbb{C}\left [\frac{G}{P} \big [\psi_H \big ]\right ] \stackrel{\cong}{\longrightarrow} \mathbb{C}\left [\frac{G}{P} \right ]^{\times_{\alpha}\mu_{\alpha}}, \ (\mu_{\alpha}=\langle \alpha, H \rangle)
\end{equation} 
is an isomorphism, where the superfix means elements fixed by the action. Therefore, it follows that
\begin{align}
\frac{G}{P} \big [\psi_H \big ]&=proj \left ( \mathbb{C}\left [\frac{G}{P} \big [\psi_H \big ]\right ] \right ) \\ \nonumber
&=\left [ \times_{\alpha \in R(P)_+} \mathbb{Z}_{\langle \alpha,H \rangle} \right ] \setminus G/P    .
\end{align} 
\end{proof}
\begin{corollary}
There is a well-defined surjective continuous map 
\begin{equation} 
\pi_H:\frac{G}{P} \to \frac{G}{P} \big [\psi_H \big ]
\end{equation} 
that is a continuous locally finite quotient map with a finite number of cyclic singularities. 
\end{corollary}
\begin{proof}
The map $\pi_H$ is nothing but the right edge of the triangle in \eqref{eq:diagram_first} or \eqref{eq:diegramcone} and was explained along with the definition \ref{def: WHS} locally. Its global definition was given in the proof of Theorem \ref{th:theorem32}.
\end{proof}
\begin{remark} \label{WPS-Triangle}
One can compare the commutative diagram \eqref{eq:diegramcone} with its well-known analogue on projective spaces:
\begin{equation}
\begin{tikzcd}
 & \mathbb{P}^n \arrow{dr}{/\mathbb{Z}_{\lambda_0} \times ... \times \mathbb{Z}_{\lambda_n}} \\
\mathbb{C}^{n+1} \diagdown 0  \arrow{ur}{/\mathbb{C}^*} \arrow{rr}{/\mathbb{C}^*} && \mathbb{P}[\lambda_0,...,\lambda_n],
\end{tikzcd}    
\end{equation}
where the right-hand edge is the $\mathbb{C}^*$-fibration of the affine cone, and the horizontal map is the weighted action on the coordinate variables. Again, the right-hand edge is the map whose existence has to be deduced. 
\end{remark}
\subsection{Basic properties of a WHS}
We wish to determine when two WHS $\frac{G}{P} \left [\psi_{H_1} \right ]$ and $\frac{G}{P} \left [\psi_{H_2} \right ]$ are algebraically isomorphic. Let us first review the same question for WPS. At the level of WPS, the isomorphism problem can be answered via the toric fan of a WPS, \cite{RT}. There is a well-known method to determine when two weighted projective spaces are isomorphic. We can associate to $\mathbb{P}[\lambda_0,...,\lambda_n]$ a fan $\Sigma$ in $\mathbb{R}^n=:N$ on the standard basis by defining 
\begin{equation}
v_i=\frac{1}{\lambda_i}e_i, \qquad e_i \ \text{ standard basis in} \ \mathbb{R}^n.
\end{equation}
One may add another vector 
\begin{equation} 
v_0=-\frac{1}{\lambda_0}(e_1+e_2+...+e_n)
\end{equation} 
to build up a fan (uniquely for the fan condition). We call it the fan associated with the WPS, [see \cite{RT} for more details]. The fan above makes the WPS a toric variety. Many of the geometric properties of the WPS can be characterized by the combinatorics of their toric fans. We first explain a reduction of weights arguments. One can reduce the weights on a WPS as follows. Define,
\begin{equation}
\begin{aligned}
&d_j=lcm(\lambda_0, \dots, \widehat{\lambda_j}, \dots , \lambda_n)\\
&a_j=gcd(d_0, \dots, \widehat{d_j}, \dots , d_n)\\
&(\lambda_0', \dots , \lambda_n')= (\lambda_0/a_0, \dots , \lambda_n/a_n),
\end{aligned}
\end{equation}
then one has $\mathbb{P}[\lambda_0, \dots , \lambda_n] \cong \mathbb{P}[\lambda_0', \dots , \lambda_n']$ [\cite{RT}, proposition 1.26], and the latter WPS is called the reduction of the first (also called reduced). The major question is to characterize when two WPS's with different weight systems are isomorphic. The answer is as follows. Two reduced WPS are isomorphic if and only if their fans are equivalent, as stated in the following theorem.
\begin{theorem} [\cite{RT} proposition 2.7] \label{th:theoremWPS}
Two reduced WPS $\mathbb{P}[\lambda_0, \dots , \lambda_n]$ and $\mathbb{P}[\lambda_0', \dots , \lambda_n']$ with the associated fans $[v_1,  \dots , v_n, v_0]$ and $[v_1', \dots, v_n', v_0']$ are isomorphic iff there exists a matrix in $GL_n\mathbb{Z}$ which exchanges their associated fans of the two in $\mathbb{R}^n$, i.e if there exists $\phi \in GL_n\mathbb{Z}$ such that $\phi v_i =v_i'$ [as explained above]. It will be the same if we just consider the transfer over the vectors $v_1,...,v_n$. 
\end{theorem}

One has an equivalent formulation via the polytopes associated with these fans [cf. \cite{RT}]. The polytope is just the convex hull of the vectors $v_i$ with the origin added. The language of fans and polytopes interchange in the toric geometry of projective varieties. If one puts the vectors $v_i$ as the columns of a matrix $[v_0,...,v_n]$, the condition above means that both of the weight matrices obtained this way has the same Hermit canonical forms, [see \cite{RT} for details]. Theorem \eqref{th:theoremWPS} above characterizes the weighted projective spaces up to isomorphism. We shall extend the argument of Theorem \ref{th:theoremWPS} to determine when two weighted homogeneous spaces of the same type $\frac{G}{P}[\psi_{H_1}]$ and $\frac{G}{P}[\psi_{H_2}]$ are algebraically isomorphic. To this end, we introduce a fan that records the weights of a WHS, and is a formal enlargement of a toric fan of the maximal torus.
\begin{definition} \label{def:extendedfan}
Assume $H \in C_+$ and $\psi_H$ be the associated weight on $\Pi$. We associate a fan (this is not a toric fan) on the coroot lattice by defining 
\begin{equation}
v_i:=\frac{1}{\langle \alpha_i, H\rangle}{\alpha_i}^{\vee}, \qquad i=1,...,n
\end{equation}
where $\alpha_i$ is a simple root. This is a fan on the coordinates of the maximal torus $T$ and we denote it by $\Sigma_H(\Pi)$. Then define a bigger fan on $\mathbb{R}^{|R(P)_+|}$, with the axis corresponding to $x_{\alpha}$, by the vectors 
\begin{equation}
v_{\alpha}=\frac{1}{\langle \alpha,H_j\rangle} e_{\alpha}, \qquad \alpha \in R_+
\end{equation}
where $e_{\alpha}$ are the standard basis of $\mathbb{R}^{|R(P)_+|}$ and add the additional vector 
\begin{equation}
v_0=-\frac{1}{\lambda_0}\sum_{\alpha \in R_+} e_{\alpha}
\end{equation}
to the $v_{\alpha}$ (uniquely) to satisfy the fan condition, [the fan condition means, the collection of $v_{\alpha}$'s together with $v_0$ satisfy a convex linear combination equal to be zero]. We denote this fan by $\Sigma_H$. 
\end{definition}
Because of the formal enlargement of the fan we applied in the definition, it fails to be a toric fan. By the way, a homogeneous space can not be toric. One can see this from a dimension criterion. The maximal torus that acts on a homogeneous space has a strictly smaller dimension than the WHS itself.

To extend the above idea over the WHS, we proceed as follows. Again we can make the weight systems on WHS reduced. Let $\psi_H$ be the weight system defined by $H \in C_+$. Again set
\begin{equation}
\begin{aligned}
&d_{\beta}=lcm \{\alpha(H) \ | \ \alpha \in J, \alpha \ne \beta\}\\
&a_{\beta}=gcd \{ d_{\alpha} \ | \ \alpha \ne \beta \}\\
&\psi_{H'}=(\alpha_H'=\alpha(H)/a_{\alpha})_{\alpha \in J},
\end{aligned}
\end{equation}
then, the same as the case of a WPS we call the weight system $\psi_{H'}$ reduced, or the reduction of $\psi_H$. 
\begin{prop} \label{thm:embedding}
Let $\frac{G}{P}[\psi_H]$ be a weighted homogeneous space. Then there exists a suitable Cartier divisor $H$ on $\frac{G}{P}[\psi_H]$ which defines a smooth embedding $(\frac{G}{P}[\psi_H], H) \hookrightarrow (\mathbb{P}_{\Delta}, \mathcal{O}(1))$ into a weighted projective space $\mathbb{P}_{\Delta}$ (associated to a polytope).   
\end{prop} 
\begin{proof} 
Consider the linear equation $\sum_{\alpha} \alpha_H'x_{\alpha}=1$. The solutions to the equation define Cartier divisors on $CH^1(\frac{G}{P}[\psi_H])$ that are also ample. That is to the solution $(\dots, b_{\alpha}, \dots)$ we associate the divisor $\sum_{\alpha}b_{\alpha}D_{\alpha}$, where $D_{\alpha}$ corresponds to the Poincar\'e dual of the differential (1,1)-form $dx_{\alpha} \wedge dx_{-\alpha}$. Define the points in $M=\mathbb{R}^n$ by 
\begin{equation}
    P_{\alpha}=(0, ..., \delta/\alpha(H), ..., 0), \qquad  \alpha \in J.
\end{equation}
Set $\delta=lcm \{\alpha(H)\ | \ \alpha \in J \}$, and let $\Delta_J$ be the $|J|$-dimensional simplex obtained as the convex hull of the origin and the points $P_{\alpha}$. The polytope $\Delta_J$ do characterizes a weighted projective space $\mathbb{P}_{\Delta_J}$, and also we have a natural map $\frac{G}{P}[\psi_H] \hookrightarrow \mathbb{P}_{\Delta_J}$ by the way of construction. 
\end{proof}
Similar to the weighted projective spaces, we can make the weight systems over a weighted homogeneous space reduced. Naturally, two WHS of the same type but with different weight systems are isomorphic. For instance, the reader may guess that the resulting WHS is the resulting WHS on a weighted root system of type $A_n$ if we make a permutation on the weights by an element of the symmetric group $S_n$ isomorphic to the previous one. In this regard, we present our following main result.
\begin{theorem}[Main Result]
The two weighted homogeneous spaces $\frac{G}{P}[\psi_{H_1}]$ and $\frac{G}{P}[\psi_{H_2}]$ are isomorphic if and only if there exists a matrix in $GL_n\mathbb{Z}$ exchanging their fans, i.e if there exists an isomorphism $\phi:\Sigma_{H_1}(\Pi) \stackrel{\cong}{\longrightarrow} \Sigma_{H_2}(\Pi)$ where $\phi \in GL_{|R(P)_+|}(\mathbb{Z})$. 
\end{theorem}
\begin{proof}
Consider the extended fan defined in Definition \ref{def:extendedfan} on $\mathbb{R}^{|R(P)_+|}$ on the standard basis labeled by the positive roots $\alpha \in R(P)$, i.e. 
\begin{equation}
v_{\alpha}^j=\frac{1}{\langle \alpha,H_j\rangle} e_{\alpha}, \qquad j=1,2,
\end{equation}
where $e_{\alpha}$ are the standard basis of $\mathbb{R}^{|R(P)_+|}$. One can also add the additional vector
\begin{equation}
v_0^j=-\frac{1}{\lambda_0}\sum_{\alpha \in R(P)_+} e_{\alpha}
\end{equation}
to the $v_{\alpha}$ (in a unique way) to satisfy the fan condition. From our definition of weighted homogeneous space it follows that if we have the isomorphism  $\Sigma_{H_1} \stackrel{\phi}{\longrightarrow} \Sigma_{H_2}$, then, the two weighted homogeneous spaces are isomorphic, i.e. $\frac{G}{P}[\psi_{H_1}] \stackrel{\cong}{\longrightarrow} \frac{G}{P}[\psi_{H_2}]$. For the converse implication assume we have an isomorphism of two WHS associated to $\psi_{H_1}$ and $\psi_{H_2}$. The isomorphism of the two WHS induces an isomorphism of 
\begin{equation} \mathcal{O}_{\frac{G}{P}[\psi_{H_1}]}(1) \cong \mathcal{O}_{\frac{G}{P}[\psi_{H_2}]}(1).
\end{equation} 
Any Cartier divisor representing these line bundles can be given as $\sum_j a_jD_j$ where $(a_1, \dots , a_{|R(P)_+|})$ is an integer point in the hyperplane 
\begin{equation} 
V: \sum_{\alpha}\lambda_{\alpha}x_{\alpha}-\delta=0, \ \ \   resp. \  V': \sum_{\alpha}\lambda_{\alpha}'x_{\alpha}-\delta'=0
\end{equation} 
where $\delta$ and $\delta'$ denotes the least common multiple of the weights]. Because these representatives are getting exchanged under the given isomorphism, it follows that any integer point of the $V$ and $V'$ can be transferred to each other. This proves the existence of the desire isomorphism $\phi \in GL_{|R(P)_+|}(\mathbb{Z})$ on lattices. Therefore, we have proved
\begin{equation} 
\frac{G}{P}[\psi_{H_1}] \stackrel{\cong}{\longrightarrow} \frac{G}{P}[\psi_{H_2}]  \Leftrightarrow  \Sigma_{H_1} \stackrel{\phi}{\longrightarrow} \Sigma_{H_2} 
\end{equation}
where $\phi \in GL_{|R(P)_+|}\mathbb{Z}$. Because of the formal enlargement of the space of lattices in the Definition \ref{def:extendedfan} and the linearity, we also have
\begin{equation}
\Sigma_{H_1} \stackrel{\cong}{\longrightarrow} \Sigma_{H_2}  \Leftrightarrow   \Sigma_{H_1}(\Pi) \stackrel{\cong}{\longrightarrow} \Sigma_{H_2}(\Pi)   
\end{equation}
exchanging $v_{\alpha}^{(i)},\ i=1,2, \ \alpha \in \Pi$. This completes the proof.
\end{proof}
Algebraic maps between homogeneous space $G/P \to G/P$ are ordinary polynomial maps on the coordinates of the projective varieties. For weighted homogeneous spaces, a regular map is given locally as
\begin{equation} \label{eq:polynomialmap}
\begin{aligned} 
f: \frac{G}{P}[\psi_{H_1}] \  &\longrightarrow \  \frac{G}{P}[\psi_{H_2}],\\
(x_{\alpha}) &\longmapsto (P_{\alpha}(\dots x_{\beta}^{m_{\alpha \beta}} \dots )), 
\end{aligned}
\end{equation}
where $P_{\alpha}$ are polynomials in the variables $x_{\beta}$, in a way that the exponents be matched with the weights. An interesting exercise is perhaps to write down examples of these maps as matrices with monomial entries. We next present this in a more precise result.
\begin{prop} [Main Result]
Assume the wight systems $\psi_{H_1}$ and $\psi_{H_2}$ are reduced. Then, there is a map 
\begin{equation} \label{eq:relativemap}
\frac{G}{P}[\psi_{H_1}] \longrightarrow  \frac{G}{P}[\psi_{H_2}]
\end{equation}
precisely when we have a homomorphism of fans
\begin{equation} 
\begin{aligned}
\Sigma_{H_1}(\Pi) &\longrightarrow \Sigma_{H_2}(\Pi)\\
v_{\alpha}^{(1)} & \longmapsto v_{\alpha}^{(2)}.
\end{aligned} 
\end{equation}
\end{prop}
\begin{proof}
As in the proof of the previous theorem a map $\Sigma_{H_1}(\Pi) \longrightarrow \Sigma_{H_2}(\Pi)$ induces also a map on the larger fans on the weights on all roots (in fact, coroots by our construction). By the construction in the Theorem \ref{thm:embedding}, this induces a map on the corresponding polytopes of the WPS's where these two WHS can be embedded. By toric geometry we obtain a well defined map between the corresponding WPS's. The restriction to $\frac{G}{P}[\psi_{H_1}]$ gives the desired map we look for. Conversely if we have a map $\frac{G}{P}[\psi_{H_1}] \longrightarrow  \frac{G}{P}[\psi_{H_2}]
$, then, by writing the coordinates out of weights, each coordinate $x_{\alpha}$ is replaced by a variable $z_{\alpha}^{\langle \alpha,H\rangle}$. Other way to explain this is to look at the induced map 
\begin{equation} \label{eq:Iso}
\mathbb{C}\left [\frac{G}{P}[\psi_{H_2}] \right ] \longrightarrow  \mathbb{C}\left [\frac{G}{P}[\psi_{H_1}]\right ]
\end{equation} 
Then the components of the map must be functions of $z_{\alpha}^{\langle \alpha,H\rangle}$. Therefore, the weights in the right hand side of this map are divisible by those in the left. This is precisely when we have $\Sigma_{H_1}(\Pi) \longrightarrow \Sigma_{H_2}(\Pi)$. 
\end{proof}
Keeping the reduced assumption on the weights, the map in \eqref{eq:relativemap} can be lifted to the affine cone over the WHS introduced in the Definition \ref{def: WHS}, 
\begin{equation} 
\begin{CD}
\left [Cone \left (\dfrac{ G}{P} \right )\diagdown 0 \right ] @>\exists>> \left [Cone \left (\dfrac{ G}{P} \right )\diagdown 0 \right ]\\
@V/\psi_{H_1}VV  @VV/\psi_{H_2}V\\
\frac{G}{P}[\psi_{H_1}] @>>f> \frac{G}{P}[\psi_{H_2}]
\end{CD}
\end{equation}
where the map on the affine cones is given by the same formula on affine charts. One can also consider a compact lift of the map $f$ as 
\begin{equation} 
\begin{CD}
G/P @>\exists>> G/P\\
@V\pi_{H_1}VV  @VV\pi_{H_2}V\\
\frac{G}{P}[\psi_{H_1}] @>>f> \frac{G}{P}[\psi_{H_2}]
\end{CD}
\end{equation}
where its existence is guaranteed by the Theorem \ref{th:theorem32}. We emphasize that the map in the upstair is given by the same formula as $f$ and it is in form \eqref{eq:polynomialmap}. Although the last theorem may influence the idea that the weighted homogeneous space behaves like toric varieties, as we mentioned for dimension reasons, neither a homogeneous space nor a weighted homogeneous space can be toric. A simple way to realize examples of the maps in \eqref{eq:Iso} is to write them as matrices with monomial entries where the exponents are compatible with the weights on different coordinates.
\section{Invariant K\"ahler forms on $\frac{G}{P}[\psi_{H}]$} \label{sec:InvKahler}
One of the significant features of the geometry of the homogeneous spaces is the study of the smooth differential forms and especially invariant differential forms on them. One way to present this notion is by using the lift of potential functions defined on $G$. This enables us to relate this notion to the structure theory of the complex reductive Lie group $G$. In this section, we propose extending the existing theorems and results on homogeneous manifolds to the case of weighted homogeneous spaces. Furthermore, we analyze how the formulas have to be modified when the coordinates are weighted. We expect this will assist us in better understanding the geometry in the weighted case.

Assume $K$ is a maximal compact subgroup of the complex reductive Lie group $G$. For an irreducible representation $\varrho$ of $G$ of highest weight $\lambda$ and highest weight vector $v_{\lambda}$ the function: 
\begin{equation}
\begin{aligned}
\phi&:G \to \mathbb{R} \cup - \infty \\
\phi&(g)=\log|\varrho(g)v_{\lambda}|
\end{aligned}
\end{equation}
is $K$-invariant, \cite{AKQ}. We may consider the highest weight representations $\rho_{\alpha}$ associated to the fundamental weights defined by 
\begin{equation}
\varrho_{\alpha}(\ \check{\beta}(z))=z^{\delta_{\alpha \beta}}    
\end{equation}
where we are using the notation in \eqref{eq:root-group}. The associated potential is defined by 
\begin{equation} 
\phi_{\alpha}(\check{\beta}(z))=\delta_{\alpha \beta}\log|z|.
\end{equation} 
The K\"ahler forms defined by the potentials $\phi_{\alpha}$ form a basis for $H^2(G/P, \mathbb{R})$. At a point $[x_0] \in G/P$ one defines the Dynkin line 
\begin{equation} 
\mathbb{P}_{\alpha}=L_{\alpha} . [x_0] \ \cong \mathbb{P}(\mathbb{C})
\end{equation} 
as homology classes where the differential 1-forms can be integrated. In \cite{AKQ}, it is shown that $\int_{\mathbb{P}_{\alpha}} w_{\beta}=\delta_{\alpha \beta}$. In order to compare this notion with the weighted one, we should consider the above forms as elements in $H^2(G/P, \mathbb{Z})$. In fact, the aforementioned classes can be calculated as Chern classes of $K$-invariant line bundles over $G/P$. 

A closed $K$-invariant $(1,1)$-form $\omega$ on $G/P$ is given by a form satisfying the identity 
\begin{equation} 
pr^* \omega=\sqrt{-1}\partial \bar{\partial} \phi,
\end{equation} 
where $\phi:G \to \mathbb{R} \cup -\infty$ is given as above, \cite{AKQ}. An application of this could be the calculation of the curvature of line bundles on the homogeneous manifolds. A simple formula on this matter is given in \cite{AKQ}. The line bundles on $G/P$ has a simple explanation in terms of the representation theory of $G$.
\begin{remark} \cite{De}
In the classical case where $G/P$ is given as a flag variety of descending subspaces $V_j, \ 1 \leq j \leq r$ of the vector space $V$, one can calculate the aforementioned K\"ahler forms as the curvature of certain determinant line bundles. The subspaces $V_j$ define tautological bundles $\mathcal{V}_j$ on $G/P$, where the fibers of $\mathcal{V}_j$ on $G/P$ are isomorphic to $V_j$. Now, consider the determinant bundle of the quotients given by, $\mathcal{V}_{j-1}/\mathcal{V}_{j}$
\begin{equation} \label{eq:Ljj}
L_j=\text{det} (\mathcal{V}_{j-1}/\mathcal{V}_{j}).
\end{equation} 
Then, the curvature of the line bundles $L_j$ can be calculated as 
\begin{equation} 
\Theta(L_j)_0 = - \sum_{i < j} dz_{ij} \wedge d \bar{z}_{ij} + \sum_{i>j}dz_{ij} \wedge d \bar{z}_{ij} ,
\end{equation} 
where $(z_{ij})_{i<j}$ are the coordinates on $U$, [cf. \cite{De}]. The line bundles on $G/P$ can be classified according to the weights $a=(a_1, \dots ,a_r) \in \mathbb{Z}^r$. In case one sets
\begin{equation}
L^a=L_1^{a_1} \otimes ... \otimes L_r^{a_r},
\end{equation}
the curvature of the line bundle $L^a$ can be calculated from
\begin{equation} 
\Theta(L^a)_e = \sum_{i<j}(a_i -a_j)dz_{ij} \wedge d\bar{z}_{ij} .
\end{equation}
Thus $L^a$ is K\"ahler if and only if $a_i > a_j, \ i<j$. 
\end{remark}
Our strategy is to study how the theorems and results in \cite{AKQ} can be stated when the coordinates have some weights. Especially the $K$-invariant K\"ahler forms can also be studied over the weighted homogeneous spaces. First, we consider the setup given at the beginning of this section. Assume we are given the weight system $\psi_H$ and the WHS, namely $\frac{G}{P}[\psi_H]$ as defined in section \ref{sec:WeightedHomogSp}. As mentioned before, the WHS is a singular variety. Thus we talk about differential forms on their smooth locus. Especially the invariant K\"ahler forms are defined only on the smooth locus of $\frac{G}{P}[\psi_H]$. Our approach to the proof of the main result below uses the context introduced at the end of Section \ref{sec:HomogSp} together with the one in \cite{Al}.
\begin{theorem}[Main Result] \label{Thm:InvariantForm}
If $\omega$ is a closed $K$-invariant $(1,1)$-form on $\frac{G}{P}[\psi_{H}]$, then its pull back over $G$ via the projection map can be written as 
\begin{equation} 
pr^* \omega=\sqrt{-1}\partial \bar{\partial} \phi,
\end{equation} 
where $pr:G \to G/P \to \frac{G}{P}[\psi_{H}]$ is the composition map and  
\begin{equation} \label{eq:weighted-potential}
\begin{aligned}
\phi&:G \to \mathbb{R} \cup -\infty\\
\phi&(g)=\sum_{J(P)} c_{\alpha}\log \| \rho_{\alpha}(g).v_{\alpha} \| .
\end{aligned}
\end{equation}
The form $\omega$ is K\"ahler iff $c_{\alpha}>0$. 
\end{theorem}
\begin{proof}
Assume $K$ is maximal compact of $G$. Fix the two root decompositions:
\begin{equation}
\mathfrak{g}=\mathfrak{h} \oplus \bigoplus_{\alpha \in R} \mathbb{C}E_{\alpha}, \qquad \mathfrak{k}=\mathfrak{h} \oplus \bigoplus_{\beta \in R_c} \mathbb{C} E_{\beta}   
\end{equation}
and let $\Pi$ and $\Pi_c$ the corresponding set of simple roots, respectively. Set $R'=R \setminus R_c,\ \Pi'=\Pi \setminus \Pi_c$. Then the complex tangent bundle of $X$ is given by 
\begin{equation}
T_0X=\sum_{\alpha \in \Pi \setminus \Pi_c}\mathbb{C} E_{\alpha}   .
\end{equation}
Let $P=LN$ be the Levi decomposition of $P$. By Theorem \ref{thm:K-G-P} we have
\begin{equation}
K \setminus G/P \stackrel{\cong}{\longrightarrow} \left ( L^{\theta}\setminus L/P \cap L \right ) \times_{W_L} W    
\end{equation}
where the fiber product is w.r.t. the cross action of $W_L$ on $L^{\theta}\setminus L/P \cap L$. Because the roots that are fixed by $\theta$ are the compact roots, i.e. the roots in $R^c$, then the roots involving in the right hand side are those in $R'$. According to [\cite{Al} page 27] there is a 1-1 correspondence between elements of $\mathfrak{h}_{\mathbb{R}}$ and the $K$-invariant forms, as
\begin{equation} \label{eq:Kahler-form}
\mathfrak{t}^* \ni  \xi=\sum c_r\tilde{\alpha}_r \ \ \stackrel{1-1}{\Leftrightarrow}\ \  \frac{i}{2\pi}\sum_{\alpha \in R_+'} \langle \alpha | \xi \rangle w^{\alpha} \wedge w^{-\alpha}   ,
\end{equation}
where $w^{\alpha}$ are 1-forms dual to $E_{\alpha}$. The pull back of the (1,1)-form in the right hand side of \eqref{eq:Kahler-form} under the map $pr$ is a (1,1)-form on $G$. Thus, according to \cite{AKQ} must be of the form \eqref{eq:weighted-potential}. We note that the potentials in \eqref{eq:weighted-potential} when expanded the differential operators are (1,1)-forms in the format of \eqref{eq:Kahler-form}. This completes the proof.
\end{proof}
The natural question is how the $K$-invariant K\"ahler forms on $G/P$ and $\frac{G}{P}([\psi_H])$ are related with the local projection $\pi_H$. It is not hard to see that the pullback of $K$-invariant K\"ahler differentials by $\pi_H$ carries a $K$-invariant k\"ahler form on $\frac{G}{P}([\psi_H])$ to a $K$-invariant k\"ahler form on $G/P$. The significance of Theorem \ref{Thm:InvariantForm} is when we consider the differential forms as representatives of cohomology classes with $\mathbb{Z}$-coefficients. In this case, the coefficients $c_{\alpha}$ are to be chosen more carefully, and their relation with the unweighted differentials on $G/P$ gets more highlighted. In this case, the coefficients $c_{\alpha}$ of the pullback form are divisible by the weight $\langle \alpha, H \rangle$. This feature shows up clearly when for instance, one calculates the Chern classes of vector bundles on $\frac{G}{P}[\psi_H]$. According to [\cite{AKQ} sec 3] the invariant line bundles on $G/P$ correspond to the characters $\chi:P \to \mathbb{C}$ of the parabolic subgroup $P$.
An invariant line bundle $L$ on $\frac{G}{P} \left [\psi_H \right ]$ can be pulled back to an invariant line bundle on $G/P$ through the finite map $\pi_H$. Line bundles on WHS are those line bundles that are invariant under the action of the finite group $\times_{\alpha} \mu_{\alpha}$. Therefore, we can also parametrize the $K$-invariant line bundles over $\frac{G}{P} \left [\psi_H \right ]$ by the characters $\chi:P \to \mathbb{C}$.
\begin{prop} [Main-Result] \label{prop:chernclass}  
Assume $L_{\chi}$ is a line bundle on $\frac{G}{P} \left [\psi_H \right ]$ corresponding to the character $\chi:P \to \mathbb{C}^*$. Then,
\begin{equation} \label{eq:chernclass}
c_1(L_{\chi})= -\sum_{\alpha \in \Pi(J)} \langle \chi, \check{\alpha} \rangle \langle \alpha, H \rangle [w_{\alpha}].
\end{equation}
In particular, the first Chern class of the canonical line bundle $K_{\psi_H}$ is equal to
\begin{equation} \label{eq:canon-bundle}
c_1(K_{\psi_H})= -\sum_{\alpha \in J} \langle \rho, \check{\alpha} \rangle \langle \alpha, H \rangle [w_{\alpha}],
\end{equation}
where $\rho$ is the sum of all positive roots in $J$.
\end{prop}
\begin{proof}  
We shall consider the differential forms on the smooth locus of $\frac{G}{P} \left [\psi_H \right ]$. Under the projection $\pi_H:\frac{G}{P} \longrightarrow \frac{G}{P} \left [\psi_H \right ]$, the pull back of an invariant K\"ahler differential is defined by the potentials \eqref{eq:weighted-potential}. The Chern class in \eqref{eq:chernclass} is of the form $\sum_{\alpha \in S \setminus \pi} c_{\alpha}[w_{\alpha}]$, where $\pi$ consists of the simple roots supporting the Levi component of $P$. The coefficients in the sum are calculated via the formula 
\begin{equation} \label{eq:int-chern}
c_{\alpha}=\int_{\mathbb{P}_{\alpha}[\psi_H]}c_1(L_{\chi})
\end{equation} 
where $\mathbb{P}_{\alpha}[\psi_H]$ are the Dynkin lines. Because the pull back of the Dynkin lines on $\frac{G}{P} \left [\psi_H \right ]$ through $\pi_H$ are the Dynkin lines of $G/P$, we need to look at  $\int_{\mathbb{P}_{\alpha}[\psi_H]}[w_{\alpha}]=\langle \alpha ,H \rangle $. The integral in \eqref{eq:int-chern} can be written as 
\begin{equation} \int_{\mathbb{P}_{\alpha}[\psi_H]}c_1(L_{\chi})=-\langle \chi, \check{\alpha} \rangle \langle \alpha, H \rangle.
\end{equation}
The last criteria in \eqref{eq:canon-bundle} follows from the fact that the canonical bundle is the line bundle associated to the character $\chi_K=\sum_{\alpha} \alpha$.
\end{proof}
The proposition \ref{prop:chernclass}, (as well as the equation \eqref{eq:weighted-potential}) can be well understood when we consider the Chern classes as cohomology classes with $\mathbb{Z}$-coefficients, i.e., $c_1(L_{\chi}) \in H^2(\frac{G}{P} \left [\psi_H \right ], \mathbb{Z})$. In this case the multiplicities appearing in the weighted case play an important role, i.e., they need to be chosen so that the class of the K\"ahler form lies in $H^2(\frac{G}{P} \left [\psi_H \right ], \mathbb{Z})$. We have prepared everything to extend the main result of \cite{AKQ} to the weighted case.
\begin{corollary}[Main-Result]\label{cor:KahEinest}
Assume $P \subset Q$. The $K$-invariant K\"ahler-Einstein metric of $G/P$ restricted to any fiber of the fibration $\frac{G}{P} \left [\psi_H \right ] \longrightarrow \frac{G}{P} \left [\psi_H \right ]$ is K\"ahler-Einstein.
\end{corollary}
\begin{proof}
The proof is a corollary of Proposition \ref{prop:chernclass} and the same argument as in \cite{AKQ}. We briefly sketch the necessary modifications. As it is explained in the last section of \cite{AKQ} the first Chern class of $\frac{G}{P}[\psi_H]$ is the Chern class of the line bundle associated to the character $\chi=\sum_{\alpha \in R_u(P)} -\langle \alpha, H \rangle \alpha$. The Levi-complement of $P$ is described by a set $\pi$ of simple roots, and $\alpha \in R_u(P)$ if and only if $\alpha$ is a positive root not supported by $\pi$. By the formula in Proposition \ref{prop:chernclass}, we have 
\begin{equation} \label{eq:1}
c_1(\frac{G}{P}[\psi_H])=-\sum_{\alpha \in \Pi \smallsetminus \pi} \langle \chi, \check{\alpha} \rangle \langle \alpha, H \rangle [w_{\alpha}].
\end{equation}
It is known that $c_1(\frac{G}{P})$ and hence $c_1(\frac{G}{P}[\psi_H])$ is positive by Proposition \ref{prop:chernclass}. Therefore, it suffices that $c_1(\frac{Q}{P}[\psi_H])=c_1(\frac{G}{P}[\psi_H])|_{\frac{Q}{P}[\psi_H]}$ and then, show that the latter is positive. Let $\tilde{\pi}$ be the simple roots of the Levi component of $Q$. Then the character of the anti-canonical bundle of $\frac{G}{P}[\psi_H]$ can be written as the sum of two characters, one supporting the positive roots with support $\Pi \smallsetminus \tilde{\pi}$ and the other positive roots with support in $\tilde{\pi} \smallsetminus \pi$. As in \cite{AKQ},
\begin{equation} \label{eq:2}
c_1(\frac{Q}{P}[\psi_H])|_{\frac{Q}{P}[\psi_H]}=-\sum_{\alpha \in \tilde{\pi} \smallsetminus \pi} \langle \chi, \check{\alpha} \rangle \langle \alpha, H \rangle [w_{\alpha}].
\end{equation}
Now, by Theorem \ref{Thm:InvariantForm}, if \eqref{eq:1} defines an invariant K\"ahler form, then, \eqref{eq:2} also defines an invariant K\"ahker form. 
\end{proof}
The proof of Corollary \ref{cor:KahEinest} is essentially based on the analogous result of proposition \ref{prop:chernclass} in \cite{AKQ} for an ordinary homogeneous space $G/P$. The only difference is that of multiplicities appearing in the formula of Chern classes. The following example shows how the ideas and the notations work.
\begin{example}
We choose an example of \cite{CG}. Lets consider $X=Sp(4)/B$ where $B$ is the standard Borel. $B$ consist of $4 \times 4$ matrices with upper left block upper triangular and lower right block lower triangular. Then $\mathfrak{n}^-=Lie(N^-)$ consists of matrices of the form   
\begin{align*} 
\mathfrak{n}^- &= \left\{ \begin{pmatrix}
0 & 0 & 0 & 0\\
x_1 & 0 & 0 & 0\\
x_2 & x_3 & 0 & -x_1\\
x_3 & x_4 & 0 & 0 
\end{pmatrix}  \right\}\ \  \stackrel{\exp}{\longrightarrow} \\ 
N^- &=\left\{ \begin{pmatrix}
0 & 0 & 0 & 0\\
x_1 & 0 & 0 & 0\\
P_1(x) & P_3(x) & 0 & -x_1\\
P_2(x) & x_4 & 0 & 0 
\end{pmatrix}  \right\}
\end{align*}
where $P_1, P_2, P_3$ can be calculated specifically by exponentiation of the matrix in the left. We have two simple roots $\{\alpha_1=\epsilon_1-\epsilon_2, \ 2\epsilon_2\}$ and $4$ positive roots $\{\alpha_1, \alpha_2, \alpha_1+\alpha_2, 2\alpha_1+\alpha_2\}$. The vectors $v_{\alpha}$ in the Theorem \ref{Thm:InvariantForm} are given by $e_1,\ e_1 \wedge e_2$ in standard basis of $\mathbb{C}^4$, with corresponding highest weight representations to be $\mathbb{C}^4, \ \bigwedge^2 \mathbb{C}^4$. Therefore a potential for the K\"ahler form can be written as:
\begin{equation}
\phi=\frac{c_1}{2\pi} \log \|g.e_1\|+\frac{c_2}{2\pi} \log\|g.(e_1 \wedge e_2) \|, \ c_1, c_2 >0
\end{equation}
expansion gives (cf. \cite{CG})
\begin{align*}
\phi(z)&= \frac{c_1}{2\pi} \log (1+|z_1|^2+|P_1(z)|^2+|P_2(z)|^2)\\ 
&+\frac{c_2}{2\pi} \log{ \left( \sum_{i<j}\det \begin{pmatrix} 
1 &0\\
z_1 & 1\\
P_1(z)&P_3(z)\\
P_2(z)&z_4
\end{pmatrix}^2 \right)}
\end{align*}
\end{example}
The previous example provides a general framework in which, in all other similar cases, one can compute the pluriharmonic potential function $\phi$ on $G$; see \cite{CG} for more computational examples.
\section{Cluster algebras} \label{sec:Prel-ClusterAlg}
This section is an overview of the classical theory of cluster algebras. The significant references are \cite{CR, FZ1, FZ2, FZ3, FZ4, FZ5, FZ6}. The expert reader may skip this section and directly go to the next. Cluster algebras appear as the coordinate ring of flag varieties. They provide a combinatorial framework to understand the geometry of homogeneous spaces. In this section, we present basic definitions and constructions of the cluster algebras of finite type. We also give an overview of the approach via quivers. In the next section, we employ a weighted version of the notions provided in this section.

\subsection{Cluster algebras by exchange graphs}
Consider a triangulation $T$ of an $(n+3)$-gon and mark the $n$ diagonal by $1,2,...,n$ and the $n+3$ sides by $n+1,...,2n+3$. The combinatorics of $T$ can be described by a $(2n+3) \times n$ adjacency matrix $\hat{B}=(b_{ij})$,
\begin{equation}
b_{ij}=\begin{cases}1, \ \ \ \ \ \text{if $i,j$ label two sides of a triangle clockwise}\\
-1, \ \    \text{similar but counter clockwise}\\
0, \ \ \ \ \ \ \text{otherwise}.
\end{cases}
\end{equation} 
Notice that $i$ is assigned to a side but $j$ to a diagonal. The principal part of the adjacencies is encoded in the $n \times n$ matrix $B=(b_{i,j})$ that encodes adjacencies of the diagonals. We describe matrix mutation operations between these triangulation by diagonal flips (in direction $k$) by $\mu_k:\hat{B} \longmapsto \hat{B}'$,
\begin{equation}
b_{ij}'=\begin{cases}-b_{ij},   \ \qquad \qquad  \ k \in \{i,j\}\\
b_{ij}+|b_{ik}|b_{kj}, \quad   k \notin \{i,j\},\ b_{ik}b_{kj}>0\\
b_{ij} \qquad  \qquad  \qquad \text{otherwise}
\end{cases}
\end{equation}
and write $\hat{B}'=(b_{ij}')=\mu_k(\hat{B}), \ \text{the same for} \ B=(b_{ij})$. The operation $\mu_k$ on the adjacency matrix $B$ corresponds to flipping the diagonal number $k$ in a fixed triangulation, that is, in two neighboring triangles, one removes the diagonal and adds the other opposite diagonal; we have $\mu_k(\mu_k(B))=B, \ (\text{the same for}\ \hat{B})$. Then, assign a variable to each $2n+3$ edges and assume by induction that we have associated a rational function to the edges. When doing a diagonal flip, the rational function associated with the removed diagonal, namely $X$, is replaced by the one $X'$ [where $X'$ is the new variable assigned to the new diagonal added in a flip step and is defined by the following equation], such that $XX'=ac+bd$, where $a,b,c,d$ are the sides of the neighboring triangles. The rational functions associated with the diagonals do not depend on the sequence connecting two triangulations. In general, the exchange relation corresponding to a flip in direction $k$ defined by a matrix mutation $B=(b_{ij})$ is given by
\begin{equation}
 X_kX_k'=\prod_{b_{ik}>0}x_i^{b_{ik}}+\prod_{b_{ik}<0}x_i^{-b{ik}}  .
\end{equation}
The process above provides a method to define certain rational functions with non-negative coefficients in the field of rational functions over some variables. We start with arbitrary variables assigned to a triangulation as above and then apply all possible mutations one by one. The cluster algebra is the ring generated by all these variables together with their mutation flips. We adjust this process in the following definitions.
\begin{definition} 
A seed is a pair $(\textbf{X},\hat{B})$ where $\textbf{X}$ is set of variables, and $\hat{B}$ is an exchange matrix.
\end{definition} 
Now we make the above description into a precise definition.
\begin{definition} (Cluster Algebra)\cite{FZ2}
A cluster algebra is a commutative ring generated in the field of rational functions in $N$ variables over $\mathbb{Q}$, by the generators from an initial seed, via the iterative mutation process as 
\begin{itemize}
\item $\underline{X}=\ \{x_1,...,x_N \}$ set of $N$ algebraically independent generators which split into disjoint sets 
\begin{equation} 
\begin{aligned}
C&=\{  x_1,...,x_n \}, \qquad  \text{(cluster variables)}\\
F&=\{ x_{n+1},...,x_N \}, \qquad  \text{(frozen variables)}\\
X&=C \cup D.
\end{aligned}
\end{equation}
\item An $N \times n$ matrix $\hat{B}=(b_{ij})$ and its principal part $B$ which is skew symmetrizable, (i.e. there exists a positive diagonal matrix $D$ such that $DBD^{-1}$ is skew symmetrizable) called exchange matrix. A seed mutation in direction $k$ transforms 
\begin{equation} 
\begin{aligned}
\mu_k&:(\underline{X},\hat{B}) \longmapsto (\underline{X}',\hat{B}')\\ \underline{X}'&=\underline{X} \setminus \{x_k\} \cup \{x_k'\}\\ \hat{B}'&=\mu_k(\hat{B}).
\end{aligned}
\end{equation}
\end{itemize}
\end{definition}
Matrix mutations preserve rank and symmerizability with the same $D$. A Seed mutation is an equivalence relation on seeds. If $S$ is an equivalence class we set $X=X(S)$ to be the union of all clusters of all seeds in $S$. The cluster algebra $A(S)$ is generated by the clusters in $X$ and the frozen variables $x_{n+1},...,x_N$ and their inverses
\begin{equation} 
A(S)=\mathbb{C}[X(S), x_{n+1}^{\pm 1},...,x_N^{\pm 1}].
\end{equation}
where $n$ is called the $rank(A)$. 

A surprising fact is that the Laurent polynomials that appear in a sequence of flips all have positive integers as coefficients. This appears as one of the essential conjectures in this area. One can prove a weaker version of this fact as follows.
\begin{theorem} (Laurent Phenomenon) \cite{FZ3}
The cluster variables are Laurent polynomials with integer coefficient (positive conjectural) in the variables $\{ x_1,...,x_N \}$. 
\end{theorem} 
We will deal only with cluster algebras of finite type, which definition is given as follows.
\begin{definition} 
A cluster algebra is of finite type if it has finitely many seeds. 
\end{definition}
Cluster algebras of finite type have a close connection with Coxeter-Dynkin diagrams of root systems. Their classification is the same as Cartan classification of root systems. In fact, if $A=(a_{ij})$ is a Cartan matrix then because the Coxeter diagram has not any loop, it is possible to decompose the index set into two disjoint subsets:
\begin{equation} 
I=I_+ \coprod I_- ,  \ (\text{the Coxeter diagram is bipartite}), 
\end{equation} 
[cf. \cite{FR} loc. cit.] and define the adjacency matrix $B$ by
\begin{equation} \label{eq:CartanExchange}
b_{ij}=\begin{cases}0 \qquad \qquad i=j\\
a_{ij} \qquad \ \ \ \ i \ne j , \ i \in I_+\\
-a_{ij} \qquad \ \ i \ne j ,\ i \in I_- .
\end{cases}
\end{equation}
The following proposition makes this connection more precise.
\begin{prop} \cite{FZ5, FR} \label{FiniteType-Cartan}
A cluster algebra is of finite type if and only if the exchange matrix is obtained from a Cartan matrix, as \eqref{eq:CartanExchange}. This criterion is also equivalent to the inequality $|b_{ij}b_{ji}| \leq 3$ on the exchange matrix $B=(b_{ij})$. 
\end{prop}
The Dynkin type of $A$ defined above, is called the cluster type. The definition in \eqref{eq:CartanExchange} is reversible, as 
\begin{equation}
a_{ij} =\begin{cases} -|b_{ij}|, \qquad i \ne j\\ \  0 \qquad \qquad \  i=j.
\end{cases} 
\end{equation}
In this case, there exists a skew symmetrizing matrix $D$ such that the matrix $DAD^{-1}$ is positive definite. 

In the following, we make the connection between the cluster variables of a cluster algebra of finite type and the combinatorics of the root systems more precise. Assume $R$ is an irreducible root system, $\{ \alpha_1,....\alpha_n \}$ is the set of simple roots, and the Cartan matrix $A$ is given. Let $\{ x_1,...,x_n \}$ be the cluster variables at a seed of the corresponding cluster algebra. Denote by $R_{\geq -1}$ the set of roots that are either positive or negative simple. The interrelation between cluster algebras of finite type and the root systems of Lie algebras and the structure theory of homogeneous manifolds is much deeper than what is mentioned above. The cluster variables and the root systems are related through a tropicalization map. We only mention several theorems in the following to motivate the idea.

\begin{theorem} \cite{FR, FZ5} \label{thm:tropical-coordinates}
The cluster variables are naturally labeled by the roots in $R_{\geq -1}$, i.e., one has a 1-1 assignment
\begin{equation}\label{eq:trop}
\alpha=c_1\alpha_1+...+c_m \alpha_n \ \stackrel{1-1}{ \longmapsto }\ \ \  x_{\alpha}=\frac{P_{\alpha}(x_1,...,x_n)}{x_1^{c_1}...x_n^{c_n}},
\end{equation}
where $P_{\alpha}$ is a polynomial. In this correspondence $x_{-\alpha_i}=x_i$. 
\end{theorem}
We may think of $x_{\alpha}$ as coordinates parametrizing a "certain" scheme. On the other hand, the left-hand side in \eqref{eq:trop} may be interpreted as corresponding tropical coordinates on a combinatorial object associated with the same scheme. Thus the theorem says how these two characterizations of the same object are related. A straightforward way to express this is that the combinatorial coordinate parametrizes the coordinates in a limit of ordinary coordinates after applying the function $\log_t$ to $x_{\alpha}$ when $t \to \infty$.

The aforementioned coordination $x_{\alpha}$ may be compared with the one we discussed in Section \ref{sec:WeightedHomogSp}. These are identical; however, one has to know that the coordinates $x_{\alpha}$ will change by the mutations. The mutations correspond to different local change of coordinates on the manifold charts of the homogeneous space $G/P$. To fix the argument, one should say that a specific parametrization of $G/P$ gives the above correspondences.

Cluster algebras lie on a bridge between geometry and combinatorics. Their theory provides a combinatorial parametrization of analogous geometric objects. As we mentioned, this exchange of terminology amounts to the exchange of the base field from a function field to a tropical semifield. In this manner, certain parametrization of geometric manifolds corresponds to combinatorial parametrization of polytopes and complexes.
\begin{definition} \cite{FR, FZ5}
The cluster complex is a simplicial complex on the set of all cluster variables, where its maximal simplices are clusters. By Theorem \ref{thm:tropical-coordinates}, the cluster complex can be labeled by the $R_{\geq -1}$ associated to some root systems. 
\end{definition} 
The following theorem mentions the duality between the cluster complex defined above and the exchange graph of the cluster algebra previously explained in this paper.
\begin{prop} \cite{FR, FZ5}
The dual graph of the cluster complex is precisely the exchange graph of the cluster algebra. 
\end{prop}
As one expects, the geometry of the homogeneous space $G/P$ is very much tied with the combinatorial Lie theory of the root lattice of $G$. The cluster algebras lie on a bridge between these two theories. Roughly speaking, one can assign to roots specific coordinates on $G/P$. We note that we specifically mean the above when we talk about some coordinates associated with root systems. This is a local correspondence. The aforementioned geometric-combinatorial analysis is closely related to the one mentioned in Section \ref{sec:WeightedHomogSp} on associating a fan to a WHS.  One may make this case more precise as follows.
\begin{prop} \cite{FR, FZ5}
The roots that label the cluster variables in a given cluster form a basis of the root lattice $Q$. The cone spanned by such roots form a complete simplicial fan in the ambient real vector space $Q_{\mathbb{R}}$ (cluster fan). The fan is the normal fan of simple $n$-dimensional convex polytope in the dual space $Q_{\mathbb{R}}^{\vee}$.
\end{prop}
\subsection{Cluster Algebras by Quivers}
Another parallel approach to cluster algebras is by quivers. First, we review some basic definitions and concepts about quivers. In the next section, we crucially employ this concept. A quiver $Q$ is an oriented graph with two vertices, 'frozen' and 'mutable.' We ignore edges connecting frozen vertices. A mutation on a vertex $z$ is defined in 3-steps as follows:
\begin{itemize}
    \item To each $x \to y \to z$ produce a new $x \to z$.
    \item Reverse all the edges incident to $X$.
    \item Remove oriented 2-cycles.
\end{itemize}
Mutations are involutions, and we can describe cluster exchanges on quivers the same as before. We first define seed variables. 
\begin{definition} A seed is $(Q,\textbf{X})$ where $Q$ is a quiver labeled by a cluster and $\textbf{X}$ is a set of variables. 
\end{definition} 
A seed mutation at $z$ replaces $Q$ by $\mu_z(Q)$ and $z$ by a new cluster variable $z'$ such that
\begin{equation}
zz'=\prod_{z \leftarrow y} y +\prod_{z \rightarrow y}y \end{equation} 
assigns formal variables associated to each vertex of a quiver; the cluster ring is generated by all the variables in a mutation equivalence class. The exchange graph of $A(Q)$ does not depend on the frozen part of $Q$. The cluster algebras, which only differ in the frozen part, are said of the same cluster type, [see \cite{FR}, \cite{FZ1}, \cite{FZ5}]. A quiver is of finite (mutation) type if and only if its mutation equivalence class consists of finitely many quivers (up to isomorphism). We have the following characterization of finiteness due Gabriel. 
\begin{theorem} (Gabriel's Theorem) \cite{Rei, FZ5} \label{FiniteType-Quiver}
A cluster algebra $A(Q)$ is of finite type iff the mutable part of its quiver at some seed is an orientation of a Dynkin diagram.  
\end{theorem}
We shall only deal with quivers of finite mutation types. The theorem makes the connection with the Coxeter diagrams more explicit. In the next section, we will present the notation of weighted quivers' notation by assigning a weight to the vertices. The main essay will be how the weights are affected by mutations. 
\begin{remark} 
A major application of quivers is in representation theory of algebras. They provide a combinatorial framework to explain representations in algebra. A representation of a quiver is an assignment of $\mathbb{C}$-vector space to each vertex and a linear map to an oriented edge. A quiver is of finite type if it has only a finite number of indecomposable representations.
\end{remark}
Our primary motivation from the theory of cluster algebras is that they appear as the coordinate rings of homogeneous spaces. We mention several results about this to make the idea more concrete.
\begin{theorem} \cite{GLS}
The coordinate ring of any partial flag variety $SL_n\mathbb{C}/P$ has a natural cluster algebra structure.
\end{theorem}
\begin{example}
Consider the Grassmanninan $Gr(k,n)$, where have the Pl\"ucker embeddings 
\begin{equation} 
\begin{aligned}
&Gr(k,n) \hookrightarrow \ \mathbb{P}(\bigwedge^k \mathbb{C}^n)\\ 
&pt=(v_1,...,v_k) \longmapsto v_1 \wedge ... \wedge v_k, \qquad (v_i \in \mathbb{C}^n).
\end{aligned}
\end{equation} 
By setting $I=\{1 \leq i_1 \leq i_2 \leq ... \leq i_k \leq n \}$ the $I$-th Pl\"ucker coordinate is given by 
\begin{equation}
\Delta_I=\det({v_i}_1,...,{v_i}_k)    
\end{equation}
the Pl\"ucker coordinates generate $\mathbb{C}[Gr(n,k)]$. One defines the $(k,n)$-diagrams by putting $n$ marked points $1, 2, 3,... , n$ clockwise on the boundary of a disc and oriented strands  
\begin{equation} 
S_i:\ i \longmapsto i+k \mod n, \qquad (1 \leq i \leq n)
\end{equation} 
with no oriented lenses, self or alternate crossings, nor triple crossing on a point [see \cite{BKM} for details]. We obtain several alternating regions whose boundaries are oriented clockwise. Denote by $i_1...i_s$ the region according to the boundary strands. The dual graph to a $(k,n)$-diagram is the cluster complex, [cf. \cite{BKM}].
\end{example}
The following theorem was proved independently by J. Scott.
\begin{theorem} \cite{Sc}
The homogeneous coordinate ring of any Grassmannian $Gr(k,n)(\mathbb{C})$ has a natural cluster algebra structure. More precisely, let $S$ be a $(k,n)$-diagram, and 
\begin{equation} 
X(S)=\{\Delta_{I(R)}\ | \ R \ \text{alternating region of} \ S\ \},
\end{equation}
then $X(S)$ is a cluster, $\mathbb{C}[Gr(n,k)]=\mathbb{C}[ X(S)]/I$. The exchange relations are generalized Pl\"ucker relations, that is $I$ is given by the Pl\"ucker relations.
\end{theorem}
More specifically, the clusters for the corresponding cluster algebra can be explicitly described by the alternating regions in a $(k,n)$-diagram as was explained in the example. Moreover, the exchange relations of the cluster algebra are given by the so-called generalized minors, wherein this case is called generalized Pl\"ucker relations.
\subsection{Generalized minors}
The last ingredient we recall from the theory of cluster algebras is the  generalized minors $\Delta_{\dot{u}w_j,\dot{v}w_j}$. They provide a framework to express generalized Pl\"ucker relations. They are defined as follows, \cite{FZ1, FZ5}. Let's write 
\begin{equation}
G_0=U_-TU, \qquad x=[x]_-[x]_0][x]_+ \ \ \text(unique)
\end{equation} 
($T$ maximal torus and $U$ unipotent) where $G_0$ the open subset of $G$ having this decomposition. Define the minor $\Delta_{uw_i,vw_i}$ as the regular function on $G$ s.t. whose restriction on $\dot{u}G_0\dot{v}^{-1}$ is given by
\begin{equation}
\Delta_{uw_i,vw_i} (x)=([\dot{u}^{-1}x\dot{v}]_0)^{w_i}, \qquad u,v \in W
\end{equation}
where $w_i$ is a fundamental weight and the exponent by a weight means the multiplicative character on the maximal torus $T$ given by \eqref{eq:Weyl-action}. These functions appear as the cluster variables in the construction of the coordinate ring of flag varieties, cf. \cite{FZ5, FZ1}. The generalized minors are well-defined nowhere vanishing functions on the double Bruhat cells. According to the Theorem \ref{th:theorem32} write
\begin{equation}
G=\coprod_{u}BuB =\coprod_vB_-vB_-.
\end{equation}
The double Bruhat cells are defined by 
\begin{equation}
G^{u,v}=BuB \cap B_-vB_-    .
\end{equation}
As a variety, $G^{u,v}$ is biregularly isomorphic to a Zariski open subset of an affine space, \cite{FZ1, BFZ}. The variety $G^{u,v}$ can be characterized by certain vanishing conditions of the generalized minors, [\cite{FZ5}, proposition 2.8]. The generalized minors also characterize the double Bruhat cells. That is, the double Bruhat cells are precisely defined by the nonvanishing of certain minors. They are naturally constructed by the mutations initiated from an original seed variable. In particular, they are included in the cluster variables. 
%
%
\section{The coordinate ring of $\frac{G}{P}[\psi_{H}]$ as a weighted cluster algebra} \label{sec:CoordRing-WHS}
In section \ref{sec:Prel-ClusterAlg}, the coordinate ring of a Homogeneous space $G/P$ was studied as a cluster algebra. In this section, we will study the coordinate ring of weighted homogeneous space i.e. $\mathbb{C}\left [\frac{G}{P} \left [\psi_H \right ]\right ]$. We shall prove that this coordinate ring is a weighted cluster algebra. In this case, the cluster algebra is of finite type, i.e., its exchange matrix is obtained from a Cartan matrice. We use a notion of weighted quivers studied by \cite{OT} to explain the coordinate ring of $\frac{G}{P}[\psi_{H}]$.

The homogeneous space $G/P$ can be embedded as a closed subset in the product $\prod_{j \in J} \mathbb{P}(L(\omega_j)^{\vee})$ where $\omega_j$ are fundamental weights of $G$ and $J=\Pi \setminus I$ is as explained in Section \ref{sec:HomogSp} [\cite{LG}, p.123]. This is the Pl\"ucker embedding. Thus $\mathbb{C}[G/P]$ is the multi-homogeneous coordinate ring, coming from this embedding. Let $\Pi_J$ denote the monoid of
dominant integral weights of the form $\lambda = \sum_{j \in J} a_j \omega_j
, (a_j \in \mathbb{N})$. Then, $\mathbb{C}[G/P]$ is a $\Pi_J$ -graded ring with a natural $G$-module structure. The homogeneous component with multi-degree $\lambda \in \Pi_J$ is an irreducible $G$-module with highest weight $\lambda$. In other words, we have $\mathbb{C}[G/P]=\bigoplus_{\lambda}L(\lambda)$.
Moreover, $\mathbb{C}[G/P]$ is generated by its subspace $\oplus_{j \in J}L(\omega_j)$. We generalize this argument in the following theorem.
\begin{theorem} [Main Result] \label{Thm:WeightedGradedRing}
The coordinate ring of $\frac{G}{P} \left [\psi_H \right ]$ is a weighted graded homogeneous ring. The field extension $\mathbb{C}(G/P) \big /   \mathbb{C}\left (\frac{G}{P} \left [\psi_H \right ]\right )$ has degree $ [\times_{\alpha \in R_+}\langle \alpha, H \rangle]$. 
\end{theorem}
\begin{proof} 
In the weighted case, we have the embedding in the product of weighted projective spaces as  
\begin{equation} \label{eq:embedding}
\frac{G}{P}\big [\psi_{H_1} \big] \varlonghookrightarrow \prod_{j \in J} \mathbb{P}\left (L(\langle \alpha_j,H \rangle w_j)\right )    
\end{equation}
where $L(\langle \alpha_j,H \rangle w_j)$ is the highest weight representation of $G$ of the highest weight $\langle \alpha_j,H \rangle w_j$, and $w_j$ are fundamental weights. Therefore, $\mathbb{C}\left [\frac{G}{P}[\psi_{H_1}]\right ]$ is graded by $\Pi_J[\psi_H]=N^J[\psi_H]$ the monoid of integral dominant weights of the form $\lambda=\sum_Ja_j\lambda_jw_j$. Therefore,
\begin{equation}
\mathbb{C}\left [\frac{G}{P}\big [\psi_{H} \big ]\right ]\  =\ \bigoplus_{\lambda \in \Pi_J[\psi_{H}]}L(\lambda).    
\end{equation}
Thus, it is generated by 
\begin{equation} 
\bigoplus_{j \in J}L(\langle \alpha_j,H\rangle w_j)
\end{equation} 
where $w_j$ are the fundamental weights. We deduce that the degree of the extension is equal to
\begin{align*}
\left [ \mathbb{C}(G/P) :  \mathbb{C}\left (\frac{G}{P} \big [\psi_H \big ] \right ) \right ] &=\left | \dfrac{\langle \bigoplus_{j \in J}L( w_j) \rangle }{\langle \bigoplus_{j \in J}L(\langle \alpha_j,H\rangle w_j) \rangle }\right | \\
&=  \prod_{\alpha \in R(P)_+}\langle \alpha, H  \rangle   .
\end{align*}
\end{proof}
The embedding above in \eqref{eq:embedding} may be compared with the simple embedding used by Abe et al.-Corti-Reid in \cite{AbM1, AbM2, CR, AKQ} where they add a positive integer $u \gg 0$ to the weights on all the coordinates. However, we have to stress that, in this case, the WHS is instead a different object. Thus we get a rather different embedding. 

In the previous section, we presented the method to define a cluster algebra of finite type from the mutation flips of the Cartan matrix of a Dynkin diagram. One may expect naturally if the same method can be applied when a system of positive weights is assigned to the nodes of the Dynkin diagram. The association of the weights is defined by a weight function $\psi_H$ for $H$ in the highest Weyl chamber. The question would be how the mutations should be applied to the weights. We may consider this problem more generally on quivers. Thus, we need a weighted version of the definition of a quiver. We do this for quivers of finite type. The combinatorial nature of the procedure suggests that; if a weight function could be defined on the vertex set of a quiver, one may be able to compute its deformation via the mutations of the quiver.   
\begin{definition} \cite{OT} \label{def:weighted-Quiver}
A weighted quiver $(Q,w)$ is a quiver $Q$ with a weight function $w  :Q  \longrightarrow \mathbb{Z}$. The mutation $\mu_k$ at the vertex $x_k$ changes the weights at the vertices as follows,
\begin{itemize}
    \item For each arrow $x_i \to x_k$ change the weight 
    \begin{equation} \label{eq:property1}
    \mu_k:w_i \longmapsto w_i+w_k 
    \end{equation} 
    that is 
    \begin{equation} \label{eq:property2}
        \mu_k(w)(x_i)=w_i+[b_{ik}]_+w_k, \ [b_{ik}]_+=\max(b_{ik},0)
    \end{equation}
    \item Reverse the sign at vertex $x_k$, i.e. \begin{equation} \label{eq:property3}
    \mu_k(w)(x_k)=-w_k.
    \end{equation}
\end{itemize}
The exchange relations are defined as follows. Label the vertices by the variables $z_i=x_i+y_i \epsilon$, where $x_i,y_i$ are usual variables and $\epsilon^2=0$. Then write the exchange relation as at $k$ by
\begin{equation} \label{eq:WeightedExchangeRelation}
zz'=\prod_{z_k \to z_j}z_j + \prod_{z_i \to z_k} z_i
\end{equation}
and the other variables remain unchanged.
\end{definition}
In weighted quivers, we assume that the vertices are labeled by $z_i=x_i+y_i\epsilon$ where $\epsilon$ is a formal variable with $\epsilon^2=0$. The variables $y_i$ are called odd variables. One of the ways that weighted quivers are used are in clusters superalgebras, \cite{Ov}. As it is shown in \cite{Ov}, the variables appearing by mutation at all possible mutable vertices are Laurent polynomials in the initial variables. The Laurent phenomenon implies that the denominators are monomials in the variables $x_i$ while the variable $y_i$ appear linearly in the numerators, \cite{OT}, \cite{OS}. All the cluster variables and their iterating mutations in a mutation class $S$ generate the cluster algebra. We call this algebra a weighted cluster algebra associated with the weighted quiver, $Q[w]$ denoted $A(Q)[w]$. We provide some examples to make the definition more understandable.
\begin{example} \cite{OT} 
We describe periodic quivers with periodic weight functions, where after a finite sequence of mutations, we get back to the original quiver with the same weight function.
\begin{itemize} 
\item One can label the vertices of a triangle or a square with weights to be 0, 1, -1, and orient the edges such that they give period 1 quivers with period 1 weight function. 
\item The primitive quiver $P_N^{(t)}$, $1 \leq t \leq N/2$ is the quiver with $n$ vertex and $n$ arrows such that the vertex $x_i$ is joined to the vertex $x_{i+t}, \mod(N)$ where the indices are from $1,...,N$. The arrow is from higher index to the lower one. The exchange matrix is given by
\begin{equation}
b_{ij}=\begin{cases}-1 \qquad j-i=t, \\ \ 1 \qquad \ \ i-j=t\\ \ 0 \qquad \ \  \text{else} \end{cases} .    
\end{equation}    
One simply finds a weight function on $P_N^{(t)}$ to be periodic of period 1, [see the ref. \cite{OT} for details and a classification].
\end{itemize}
\end{example}

In order to express our main result, we first prove some lemmas that connect the discussion of the weight system $\psi_H$ in Section \ref{sec:WeightedHomogSp} to that of cluster algebras and the Definition \ref{def:weighted-Quiver}. %
\begin{lemma} \label{lemma2}
Let $\frac{G}{P}[\psi_H]$ be a WHS defined by the weight system $\psi_H: \Pi \to \mathbb{Z}_+$ via the element $H \in C_+$ in the highest Weyl chamber. Consider $\psi_H$ as a weight function on the quiver $Q$ obtained by choosing an arbitrary orientation of the Dynkin diagram of $G$. Then, the mutations exchange the $\psi_H$ according to its transform under change of coordinates explained in \eqref{eq:Weyl-action}. 
\end{lemma}
\begin{proof}
The only things to be checked are if the properties \eqref{eq:property1}, \eqref{eq:property2} and  \eqref{eq:property3} hold when we do the change of coordinate on charts. We have to note that the change of coordinates from a chart that $x_{\alpha} \ne 0$ to another where $x_{\beta} \ne 0$, is given by multiplication $ \times \frac{x_{\beta}}{x_{\alpha}}$ analogous to the WPS case. However, the scalar will act on the variables according to their weights. The $T$-action gets twisted as $t._H\ x_{\alpha}^w=\ t^{w( \alpha )} \ x_{\alpha}^w, \ \alpha \in R_+, \ x_{\alpha}^w \in w.U$ where $x_{\alpha}^w$ states the coordinate $x_{\alpha}$ in $w U$. Because the Weyl group is generated by fundamental reflections, it suffices we check the claim of the lemma just for the application of fundamental reflections $r_i$ corresponding to simple roots. In this case $w(\alpha)=r_i(\alpha)=\alpha-\langle \alpha , \alpha_i^{\vee} \rangle \alpha_i$. Applying the both sides to, $H \in C_+$ we see that the properties \eqref{eq:property1}, \eqref{eq:property2} and \eqref{eq:property3} hold for the case of fundamental reflections, which finish the proof. 
\end{proof}
The following proposition encounters a new series of examples for weighted cluster algebras by putting all the tools we introduced above. In the following proposition, we have assumed that coordinate rings of the homogeneous space $G/P$ is a cluster algebra, and based on that, we explain that the coordinate rings of a WHS are a weighted cluster algebra defined by a weighted quiver, \ref{def:weighted-Quiver}. 
\begin{prop}[Main Result]
The coordinate ring $\mathbb{C}\left [\frac{G}{P}[\psi_{H}]\right ]$ is a weighted cluster algebra of finite type. Besides, it can be characterized by the cluster algebra of the weighted quiver obtained from an orientation of the Dynkin diagram of $G$ and the weight function $w=\psi_H$, together with the mutations defined as in Definition \ref{def:weighted-Quiver}.
\end{prop}
\begin{proof} (sketch)
That the coordinate ring of a WHS is a weighted graded homogeneous ring was proved in Theorem \ref{Thm:WeightedGradedRing}. It remains to check the cluster algebra conditions. By Definition \ref{def: WHS}, Theorem \ref{FiniteType-Cartan} and Theorem \ref{FiniteType-Quiver} the cluster algebra of $\mathbb{C}\left [\frac{G}{P}[\psi_{H}]\right ]$ is obtained in the same way as $\mathbb{C}\left [\frac{G}{P}\right ]$ constructed from the mutations on the Dynkin diagram of $G$, when the nodes of the diagram are assigned some weights. This shows that if the cluster algebra associated to $G/P$ is of finite type then that for the $\mathbb{C}\left [\frac{G}{P}[\psi_{H}]\right ]$ is also of finite type [this matter is independent of the weights according to definitions]. The generalized minors are defined on $\dot{u}G_0\dot{v}^{-1}$ by the same formula, $\Delta_{uw_i,vw_i} (x) = ([\dot{u}^{-1}x\dot{v}]_0)^{w_i}, \ u,v \in W$. The definition \ref{def:weighted-Quiver} shows that after the choice of weights at a seed the weight of the minors changes as the weights of the coordinates for the matrix $x_{\alpha}$ were introduced in \eqref{eq:torus-action}. The generalized minors appear as cluster variables, and the exchange relations of the mutations at any node are standard determinant identities. Thus, the coordinate rings of the double cosets $\mathbb{C}[G^{u,v}]$ are generated by the weighted generalized minors. So the Theorem is a consequence of the Theorem in the unweighted case and Lemma \ref{lemma2}.
\end{proof}
We give several examples of weighted homogeneous spaces and compute their coordinate rings as a weighted cluster algebra.
\begin{example}
\begin{itemize}
    \item[(1)] Let $\mathcal{A}=\mathbb{C}[SL_2 \mathbb{C}]$. It is a cluster algebra of rank 1. Setting the entries of a matrix by $x_{11}, x_{12}, x_{21},x_{22}$, the variables $x_{11}, x_{22}$ are cluster variables. The ring $\mathcal{A}$ is generated over $\mathbb{C}[x_{12}, x_{21}]$ by the cluster variables with the exchange relation 
    \begin{equation}
        x_{11}x_{22}-x_{12}x_{21}=1.
    \end{equation}
    Now consider the homogeneous space $X=SL_2/B^-$. We can choose the weights by assigning two positive integers $n > m$ to the torus variables $x_{11}, x_{22}$ respectively. The Dynkin diagram consists of a single point with the weight $n-m>0$. We have just one positive root $\alpha=\epsilon_1-\epsilon_2$ and one root coordinate  $x_{\alpha}$ with weight $2(n-m)$. The Weyl group is $S_2$. The coordinate $x_{\alpha}^w$ in the other Bruhat cell, for the nontrivial element $w \in W$, has the same weight as $x_{\alpha}$. We can also regard the Dynkin diagram as a trivial quiver, with weight $n-m$. The mutation changes its sign and back. 
    \item[(2)] Consider $X=SL_3/N$ and the Cluster algebra $\mathcal{A}=\mathbb{C}[SL_3/N]$, where $N$ is the maximal upper triangular unipotent subgroup. It is also a cluster algebra of rank 1. Denote the coordinates on $X$ by $x_1, x_2, x_3, x_{12}, x_{13}, x_{23}$. Then, we can choose $x_2, x_{13}$ as cluster variables and $\mathcal{A}$ is generated by $x_2, x_{13}$ over $\mathbb{C}[x_1, x_3, x_{12}, x_{23}]$ with the exchange relation
    \begin{equation}
    x_2x_{13}=x_1x_{23}+x_3x_{12}.
    \end{equation}
    We have two simple roots $\alpha_1, \alpha_2$ and 3 positive roots $\alpha_1, \alpha_2, \alpha_1+\alpha_2$. The weights can be chosen independently on the simple roots. We can illustrate this by choosing three positive weights $w_1> w_2> w_3> 0$
    on the variables $x_1, x_2, x_3$. Then the weights on the other coordinates are $w_1+w_2, w_2+w_3, w_3+w_1$. The above exchange relation still holds with the new weights. %
    \item[(3)] Consider $Gr(2,5)$, the Grassmannian of 2-planes in $\mathbb{C}^5$. The affine cone of $Gr(2,5)$ embedds in $\bigwedge^2 \mathbb{C}^5$. Thus the coordinates on $Cone(Gr(2,5))$ can be explained via the formal map
    \begin{eqnarray}
        \bigwedge^2 \begin{pmatrix}x_1 & x_2 & x_3 & x_4 &x_5\\ y_1& y_2& y_3& y_4& y_5 \end{pmatrix}= \\ 
        \begin{pmatrix} z_{11}& z_{13}& z_{14}& z_{15}\\  0 & z_{23}& z_{24}& z_{25} \\ 0 & 0 & z_{34}& z_{35} \\ 0& 0& 0& z_{45} \end{pmatrix} \nonumber
    \end{eqnarray} 
    where $z_{ij}= x_iy_j-x_jy_i$. The coordinate ring of the cone is generated by the variables $z_{ij}$ and the relations 
    \begin{equation}
        P_{ijkl}=z_{ij}z_{kl}-z_{ik}z_{jl}+z_{il}z_{jk}.
    \end{equation}
    If we assign the weights $w_i$ on $x_i$ and $y_i$, then the weight on $z_{ij}$ is $w_i+w_j$. The affine charts on $Gr(2,5)$ are given by $z_{ij} \ne 0$. On the weighted Grassmanninan the chart with $z_{ij} \ne 0$ is the quotient $\mathbb{C}^6/ \mu_{wt(z_{ij})}$. 
\end{itemize}
\end{example}
%
\section{Conclusion}\label{sec:conc}
A new approach to define a weighted homogeneous space (WHS) is provided. Specifically, we state that a weighted homogeneous variety can be written as a whole compact quotient of a homogeneous space by a finite abelian group with respect to a particular action of the maximal torus. 

Furthermore, several known results on invariant K\"ahler differentials over a homogeneous space are extended to the WHS case. 

Finally, we explain that the coordinate ring of a WHS is a weighted cluster algebra via the weighted quiver obtained from the Dynkin quiver of the associated Lie group.

\bibliographystyle{hindawi_bib_style}

\end{document}